\documentclass[a4paper,leqno,11pt]{amsart}

\usepackage{ulem}
\usepackage{enumerate}
\usepackage{amssymb,amsthm,amsmath}
\usepackage{mathrsfs}
\usepackage[colorlinks,breaklinks]{hyperref}
\usepackage{braket}
\usepackage{comment}
\usepackage{tikz}
\usepackage{pgfplots}
\pgfplotsset{compat=1.11}
\usepgfplotslibrary{fillbetween}
\usetikzlibrary{intersections}

\usetikzlibrary{positioning,arrows}
\usepackage{capt-of}
\usepackage{caption}
\usepackage{float}
\usetikzlibrary{shapes}
\usetikzlibrary{plotmarks}
\tikzset{
  state/.style={circle,draw,minimum size=6ex},
  arrow/.style={-latex, shorten >=1ex, shorten <=1ex}}

\theoremstyle{plain}
\newtheorem{prop}{Proposition}[section]
\newtheorem{lemma}[prop]{Lemma}
\newtheorem{theorem}[prop]{Theorem}
\newtheorem{cor}[prop]{Corollary}

\newtheorem*{theorem-A*}{Theorem A}

\theoremstyle{definition}

\theoremstyle{remark}
\newtheorem{remark}[prop]{Remark}

\newcommand{\C}{\mathbb{C}}
\newcommand{\R}{\mathbb{R}}

\newcommand{\CC}{\mathbb{C}}

\newcommand{\pmat}[4]{\begin{pmatrix} #1&#2\\#3&#4\end{pmatrix}}

\DeclareMathOperator{\diag}{diag}

\DeclareMathOperator{\tr}{tr}

\DeclareMathOperator{\dvol}{dvol}
\DeclareMathOperator{\vol}{vol}

\DeclareMathOperator{\id}{id}

\newcommand*{\defeq}{\stackrel{\text{def}}{=}}

\allowdisplaybreaks

\makeatletter
\@namedef{subjclassname@2020}{%
  \textup{2020} Mathematics Subject Classification}
\makeatother

\usetikzlibrary{patterns}

\begin{document}

\title[]{Uniform resonance free regions for convex cocompact hyperbolic surfaces and expanders}

\author[L.\@ Soares]{Louis Soares}
\email{louis.soares@gmx.ch}

\subjclass[2020]{Primary: 58J50, 11M36, Secondary: 37C30, 37D35, 05C90}
\keywords{resonances, hyperbolic surfaces, expander graphs, Selberg zeta function, transfer operators}
\begin{abstract}
We prove that every family of coverings of any infinite-area, convex cocompact hyperbolic surface has uniform spectral gap, provided that the associated Schreier graphs form a family of two-sided expanders. This extends the results of Brooks \cite{Brooks}, Burger \cite{Burger1, Burger2}, and Bourgain--Gamburd--Sarnak \cite{BGS} to a setting where the Laplacian has no $L^2$-eigenvalues. In particular, the notion of spectral gap needs to be redefined in terms of the resonances of the Laplacian. As an immediate corollary, we obtain uniform spectral gap for congruence covers of convex cocompact surfaces, a result previously established by Oh--Winter \cite{OhWinter} and Bourgain--Kontorovich--Magee \cite{BKM}. 

Moreover, given any convex cocompact hyperbolic surface $X$, we provide a new ``universal'' resonance-free region for $X$, by which we mean a region in the complex plane that contains no resonances for any finite cover of $X$. This enlarges the universal resonance-free region given by Magee--Naud \cite{NaudMagee}. Our methods rely on the thermodynamic formalism for twisted Selberg zeta functions.
\end{abstract}
\maketitle

%

\section{Introduction}\label{sec:intro}

\subsection{Spectral gap and expanders}\label{sec:intro0}

Let $X$ be a compact Riemannian manifold and let $(X_n)_{n\in \mathbb{N}}$ be family of finite-degree coverings of $X$. Let $\lambda_1(X_n)$ denote the first non-zero eigenvalue of the Laplacian on $X_n$. By Brooks \cite{Brooks} and Burger \cite{Burger1, Burger2}, the asymptotic behaviour of $\lambda_1(X_n)$ as $n\to \infty$ is largely determined by a combinatorial property of the fundamental groups of the coverings $X_n$. This property can be described using expander graphs, or expanders for short. There is a vast literature on expanders and their applications in computer science and pure mathematics, see for instance \cite{HLW06,Lubotzky2,Sar04,KowalskiExp}. Roughly speaking, expanders are sparse but highly connected graphs. We will adopt a ``spectral'' point of view: a finite $k$-regular graph $\mathcal{G}$ is called
\begin{itemize}
\item an ``$\epsilon$-expander'' if all the eigenvalues of the adjacency operator $A_{\mathcal{G}}$, other than the trivial eigenvalue $k$, lie in the interval $[-k, k(1-\epsilon)]$, and
\item a ``two-sided $\epsilon$-expander'' if all the eigenvalues of $A_{\mathcal{G}}$, other than $k$, lie in the interval $[-k(1-\epsilon), k(1-\epsilon)]$.
\end{itemize}
A family of $k$-regular graphs $(\mathcal{G}_n)_{n\in \mathbb{N}}$ is called a family of ``(two-sided) expanders'' if there exists some $\epsilon>0$ such that each $\mathcal{G}_n$ is a (two-sided) $\epsilon$-expander.

Let us now describe the graphs of interest to us. Fix a Riemannian manifold $X$, let $\Gamma = \pi_1(X)$ be its fundamental group, and let $S$ be a fixed set of generators of $\Gamma$. Let $X_n$ be a family of finite covers of $X$. The fundamental groups of $X_n$, denoted by $\Gamma_n$, can be viewed as subgroups of $\Gamma$. For each $n\in \mathbb{N}$ let $\mathcal{G}_n$ be the ``Schreier coset graph'' $\mathcal{G}(\Gamma,\Gamma_n,S)$. It is built as follows: the vertices of $\mathcal{G}_n$ are the left cosets of $\Gamma/\Gamma_n$ (of which there are finitely many) and two vertices of $\mathcal{G}_n$ are joined by an edge if and only if the corresponding cosets differ by left multiplication by an element in $S \cup S^{-1}$. The aforementioned results of Brooks \cite{Brooks} and Burger \cite{Burger1,Burger2} tell us  the following:

\begin{theorem-A*}
Let $X$ be a compact Riemannian manifold and let $(X_n)_{n\in \mathbb{N}}$ be family of finite-degree coverings of $X$. The following conditions are equivalent:
\begin{enumerate}[(1)]
\item The graphs $\mathcal{G}_n$ form a family of expanders. 
\item The Laplacians on $X_n$ have a uniform positive spectral gap.
\end{enumerate}
\end{theorem-A*}
Here ``uniform spectral gap'' means that there exists some constant $c>0$ such that for all $n\in \mathbb{N}$,
$$
\lambda_1(X_n) \geqslant \lambda_0(X_n) + c,
$$
where $0 \leqslant \lambda_0(X) < \lambda_1(X) \leqslant \lambda_2(X) \leqslant \cdots$ denote the eigenvalues of the Laplacian on $X$. Bourgain--Gamburd--Sarnak \cite{BGS} extended this result to all geometrically finite hyperbolic quotients $X = \Gamma\backslash\mathbb{H}^2$ whose critical exponent $\delta$ exceeds $ \frac{1}{2}$. We note that for compact manifolds the base eigenvalue $\lambda_0(X)$ is equal to zero, while for infinite-area hyperbolic surfaces with $\delta > \frac{1}{2}$ we have $\lambda_0(X)=\delta(1-\delta)$).

In this paper, we consider the case where $X=\Gamma\backslash\mathbb{H}^2$ is an \textit{infinite-area}, \textit{convex cocompact} hyperbolic surface. In this case, $X$ can be decomposed into a compact surface $N$ with geodesic boundary\footnote{The surface $N$ is sometimes called the ``Nielsen region''} and a finite number of funnel ends attached to $N$. This includes all hyperbolic surfaces with $\delta \leqslant \frac{1}{2},$ see Remark \ref{rem:secrem} below.

Let us first review some standard facts about the spectral theory of infinite-area hyperbolic surfaces, referring the reader to Borthwick's book \cite{Borthwick_book} for a comprehensive discussion. Every hyperbolic surface $X$ can be realized as a quotient $\Gamma\backslash \mathbb{H}^2$, where $\mathbb{H}^2$ is the hyperbolic plane and $\Gamma\subset \mathrm{PSL}_2(\mathbb{R})$ is a Fuchsian group, i.e., a discrete group of M\"{o}bius transformations. For the rest of this introduction we assume that $\Gamma$ is
\begin{itemize}
\item \textit{finitely generated}, which is equivalent to $X=\Gamma \backslash \mathbb{H}^2$ being geometrically and topologically finite, and
\item \textit{non-cofinite}, which is equivalent with $X$ having infinite area.
\end{itemize}
The limit set of $X$, denoted by $\Lambda$, is defined as the set of accumulation points of all orbits of the action of $\Gamma$ on $\mathbb{H}^2$. Under the above conditions, $\Lambda$ is a Cantor-like fractal subset of $\partial \mathbb{H}^2 \cong \mathbb{R}\cup \{ \infty \}$. The Hausdorff dimension of $\Lambda$, denoted by $\delta$, is equal to the exponent of convergence of the Poincar\'{e} series for $\Gamma$ and is sometimes referred to as the ``critical exponent'' of $X$.

The $L^2$-spectrum of the positive Laplacian $\Delta_X$ on $X$ is rather sparse and has been described completely by Lax--Phillips~\cite{LP1} as follows:
\begin{itemize}
\item The continuous spectrum of $\Delta_X$ coincides with the interval $[1/4, \infty)$. It does not contain any embedded $L^2$-eigenvalues.
\item For $\delta > \frac{1}{2}$ the pure point spectrum is a finite set of points in the interval $\left[ \frac{1}{4}, \delta(1-\delta) \right].$
\item For $\delta \leqslant \frac{1}{2}$ the pure point spectrum is empty.
\end{itemize}
This description of the spectrum shows that the resolvent operator
$$
R_X(s)=\left(\Delta_X-s(1-s) \right)^{-1}:L^2(X)\rightarrow L^2(X)
$$
is well-defined and analytic on the half-plane $\mathrm{Re}(s)> \frac{1}{2}$, except at the finite set of poles corresponding to the pure point spectrum of $\Delta_X$. In the infinite-area case, the main spectral data of interest are the \textit{resonances} of $X$. These are defined as the poles of the meromorphic continuation of
\begin{equation}\label{resolvent}
R_X(s):C_0^\infty(X)\rightarrow C^{\infty}(X)
\end{equation}
to the whole complex plane. The continuation of \eqref{resolvent} can be deduced from the analytic Fredholm theorem using a suitable parametrix provided by Guillop\'e--Zworski \cite{GuiZwor2}. In the sequel, we denote by $\mathcal{R}_X$ the multiset of resonances of $X$; the multiplicity of a resonance $\zeta$ is defined as the rank of the residual operator of $R_X(s)$ at $s=\zeta$. By a a result of Patterson \cite{Patterson_2}, $R_X(s)$ has a simple pole at $s=\delta$ and no other poles in the half-plane $\mathrm{Re}(s)\geqslant \delta$. In other words, the point $s=\delta$ is the resonance for $X$ with the largest real part. 

Naud \cite{Naud2} (see also Bourgain--Dyatlov \cite{Bourgain_Dyatlov}) proved that every non-elementary, geometrically finite hyperbolic surface $X = \Gamma\backslash \mathbb{H}^2$ has a ``spectral gap'' in the following sense: there is some $\eta>0$ so that
\begin{equation}\label{nauds_result}
\mathcal{R}_X \cap \{ \mathrm{Re}(s) \geqslant \delta - \eta \} = \{ \delta \},
\end{equation}
or equivalently,
\begin{equation}\label{def:gap}
\mathrm{gap}(X) \defeq \inf_{\text{$ \zeta \in \mathcal{R}_X\setminus \{ \delta \}$}} (\delta-\mathrm{Re}(\zeta)) \geqslant \eta.
\end{equation}
When $\delta > \frac{1}{2}$, this is an easy consequence of the following facts: every resonance $s$ in the half-plane $\mathrm{Re}(s) > \frac{1}{2}$ is a real number in the interval $[\frac{1}{2}, \delta]$, it corresponds to an $L^2$-eigenvalue $\lambda = s(1-s)$, and there are only finitely many such eigenvalues. In the alternative case $\delta \leqslant \frac{1}{2}$, Naud's result is far less trivial because there are infinitely many resonances in the half-plane $\mathrm{Re}(s) < \frac{1}{2}$ (as established in \cite{GuiZwor2}), none of which is connected to an $L^2$-eigenvalue.

Spectral gaps have a long history and numerous applications. In the context of infinite-area hyperbolic surfaces, these applications include precise asymptotics of the length spectrum \cite{Naud_asymptotics}, asymptotic of waves \cite{Naud_Guillamou}, as well as arithmetic applications in conjunction with the affine sieve \cite{BGS}.

\subsection{Statement of results}
Let $X'$ be a finite cover of $X$, i.e., a finite-degree or finite-sheeted cover. Any resonance for $X$ is also one for $X'$. This follows directly from the fact that resonances occur as zeros of the Selberg zeta function and from the Venkov--Zograf formula, see §\ref{sec:VZ}. Moreover, the limit sets of $X'$ and $X$ have equal Hausdorff dimensions. In particular, the critical exponent $s=\delta$ is a common resonance for all finite covers of $X$ and it is the one with the largest real part. Following Magee--Naud \cite{NaudMagee}, we say that a resonance for $X'$ is ``new'' with respect to the base surface $X$ if it occurs with greater multiplicity than in $X$. In this view, we can define the \textit{relative} spectral gap of a covering $X'\to X$ by
\begin{equation}\label{def:relgap}
\mathrm{gap}(X',X) = \inf_{\text{$ \zeta $ new}} (\delta-\mathrm{Re}(\zeta)),
\end{equation}
where the infimum is taken over the new resonances of $X'$ with respect to $X$. Clearly, we have
$$
\mathrm{gap}(X') = \min\{ \mathrm{gap}(X',X), \mathrm{gap}(X) \}.
$$

We now focus on the case where $X$ is an infinite-area, convex cocompact hyperbolic surface. Thanks to a result of Button \cite{Button}, any such surface is isometric to a quotient $\Gamma\backslash\mathbb{H}^2$ where $\Gamma$ is a so-called \textit{Schottky group}. Schottky groups stand out, among other Fuchsian groups, by their simple geometric construction, which we recall below in §\ref{Schottky_groups}. For the rest of this introduction,
\begin{itemize}
\item $\Gamma$ is a non-elementary\footnote{this is equivalent with $\Gamma$ having at least $m\geqslant 2$ Schottky generators} Schottky group with Schottky generating set $S = \{ \gamma_1,\dots, \gamma_m\}$,
\item $(\Gamma_n)_{n\in I}$, $I\subseteq \mathbb{N}$ is a family of finite-index subgroups of $\Gamma$,
\item we put $X=\Gamma\backslash\mathbb{H}^2$ and $X_n = \Gamma_n\backslash\mathbb{H}^2$, and
\item we let $(\mathcal{G}_n)_{n\in I}$ be the associated family of Schreier coset graphs with respect to the generating set $S$.
\end{itemize}

\begin{theorem}[Main Theorem]\label{main_theorem}
Let notations and assumptions be as above. Suppose the graphs $\mathcal{G}_n$ form a family of two-sided $ \epsilon $-expanders for some $\epsilon > 0$. Then there exists some $\eta = \eta(\epsilon,\Gamma) > 0$ such that
$$
\inf_{n\in I} \mathrm{gap}(X_n,X)\geqslant \eta. 
$$
Moreover, we may take $\eta = \exp(-c\epsilon^{-1})$ for some $c=c(\Gamma)>0.$
\end{theorem}

A few remarks are in order.

\begin{remark}
If $X$ has critical exponent $\delta$ larger than $\frac{1}{2}$, Theorem \ref{main_theorem} is easily deduced from Theorem 1.2 of Bourgain--Gamburd--Sarnak \cite{BGS} (albeit without the dependence of $\eta$ on $\epsilon$). However, their theorem is a statement only about eigenvalues. As such, it tells us nothing about the case $\delta \leqslant \frac{1}{2}$ since resonances in the half-plane $\mathrm{Re}(s)\leqslant \frac{1}{2}$ are not connected to eigenvalues. Moreover, the methods in \cite{BGS} do not easily generalize to our setting and we are forced to devise our own proof strategy.
\end{remark}

\begin{remark}\label{rem:secrem}
We know from Beardon \cite{Beardon_exponent1,Beardon_exponent2} that the condition $\delta \leqslant \frac{1}{2}$ excludes cusps. In other words, any hyperbolic surface with $\delta\leqslant\frac{1}{2}$ must be convex cocompact, so this case automatically comes under the purview of Theorem \ref{main_theorem}.
\end{remark}

\begin{remark}
Theorem \ref{main_theorem} fails if $\Gamma$ is elementary, i.e., when $m=1$. In this case $\Gamma$ is generated by a single hyperbolic element $\gamma$ and $X$ is a hyperbolic cylinder. The critical exponent of $X$ is equal to zero and $\mathcal{R}_X$ is equal to the half-lattice $(2\pi i /\ell)\mathbb{Z}-\mathbb{N}_0$, where $\ell >0$ is the displacement length of $\gamma$. Thus, for all elementary $X$ we have $\mathrm{gap}(X)=0$. One can verify easily that any finite cover $X'$ of an elementary surface $X$ is again elementary and therefore $\mathrm{gap}(X')=0$ as well. In particular, this shows that the conclusion of Theorem \ref{main_theorem} is false in the elementary case. 
\end{remark}

A quotient $X' = \Gamma'\backslash \mathbb{H}^2$ is said to be a \textit{regular} covering of $X = \Gamma\backslash \mathbb{H}^2$ if $\Gamma'$ is a \textit{normal} subgroup of $\Gamma$. In this case, the quotient $G \defeq \Gamma/\Gamma'$ forms a group, the so-called \textit{covering group}, and the Schreier coset graph $\mathcal{G}$ is the Cayley graph $\mathrm{Cay}(G,\tilde{S})$ of $G$ with respect to the image $\tilde{S}$ of $S$ under the natural projection map $\Gamma\to G$.

It is well-known that a connected, regular graph is bipartite if and only if its normalized adjacency matrix has $-1$ as an eigenvalue. For Cayley graphs we know much more from \cite[Appendix E]{BGGT}: if a Cayley graph is a (one-sided) $\epsilon$-expander graph for some $\epsilon > 0$, and is not bipartite, then it is a two-sided $\epsilon'$-expander for some $\epsilon' > 0$ depending solely upon $ \epsilon$ and $\vert S\vert$. This immediately leads to the following:

\begin{cor}\label{cor1}
Let notations and assumptions be as above. Assume additionally that $(X_n)_{n\in I}$ is a family of regular coverings such that each $\mathcal{G}_n$ is non-bipartite. Then the following holds true. If $\mathcal{G}_n$ form a family of (one-sided) expanders, there exists some $\eta>0$ such that $\mathrm{gap}(X_n)\geqslant \eta $ for all $n\in I.$
\end{cor}

This corollary can also be rephrased in terms of Property ($\tau$), see §\ref{sec:Schreier}.

\begin{cor}
Assumptions being as in Corollary \ref{cor1}, the following holds true. If $\Gamma$ has Property ($\tau$) with respect to the family $(\Gamma_n)_{n\in I}$, there exists some $\eta>0$ such that $\mathrm{gap}(X_n)\geqslant \eta $ for all $n\in I.$
\end{cor}

\begin{remark}
There are several conditions that guarantee a Cayley graph $\mathrm{Cay}(G,S)$ to be non-bipartite. Here are three of them:
\begin{enumerate}[(1)]
\item \label{nb:part1} There exists no surjective group homomorphism $\xi \colon G \to \{ \pm 1 \}$ with the property that $\xi(s) = -1$ for all $s\in S$. In fact, this is equivalent with $\mathrm{Cay}(G,S)$ being non-bipartite, see \cite[Proposition 2.3.6]{KowalskiExp}. 
\item \label{nb:part2} $S$ contains the identity $e_G$ of $G$, or
\item \label{nb:part3} $G$ has no non-trivial, one-dimensional representation.
\end{enumerate}
One can show easily that Properties \eqref{nb:part2} and \eqref{nb:part3} follow from the first one.
\end{remark}

Our next result deals with ``universal'' resonance-free regions for $X$, by which we mean regions $\mathcal{U}\subset \mathbb{C}$ with the property that for any finite-degree covering $X'\to X$ there are \textit{no} resonances for $X'$ inside $\mathcal{U}$. Note that the half-plane $\mathcal{U} = \{ \mathrm{Re}(s) > \delta \}$ trivially satisfies this property since the critical exponent $\delta$ does not change by passing to finite covers. Recently, Magee--Naud \cite{NaudMagee} proved that there are $\eta_\infty > 0$ and $T > 0$, depending only on $X$, such that
\begin{equation}
\mathcal{U} = \{ s = \sigma + it : \text{$\sigma > \delta -\eta_\infty$ and $\vert t\vert > T$} \}
\end{equation}
is a universal resonance-free region for $X$. Our next theorem strengthens this result.
\begin{theorem}[Universal resonance-free region]\label{thm:dist_zero_free}
For every non-elementary, geometrically finite hyperbolic surface $X$ with critical exponent $\delta$ there exists a constant $\eta_\infty > 0$ and a function $\eta\colon \mathbb{R} \to \mathbb{R} $ satisfying
\begin{itemize}
\item $\eta(0) = 0$,
\item $\eta(t) > 0$ for $t\neq 0$, and
\item $\lim_{\vert t\vert \to \infty} \eta(t) = \eta_\infty$
\end{itemize}
such that
\begin{equation}\label{region_W}
\mathcal{U} = \{ s = \sigma + it \in \mathbb{C} : \sigma > \delta -\eta(t) \}
\end{equation}
is a universal resonance-free region for $X$. Moreover, we may take
$$
\eta(t) = \eta_\infty \frac{\xi(t) }{1+\xi(t)}
$$
with $\xi(t) = \frac{\vert t\vert^2}{\log \left(  C + \frac{1}{\vert t\vert} \right)}$ for some constant $C>1$ depending only on $X$.
\end{theorem}

\begin{figure}[h]
\centering
\captionsetup{justification=centering}
\begin{tikzpicture}[xscale=1.3, yscale=1.3]
\draw [->] (-0.5,0) -- (3.05,0) node [right, font=\tiny]  {$\mathrm{Re}$};

\draw [->] (0,-3) -- (0,3) node [right, font=\tiny]  {$ \mathrm{Im}$};

\draw[dashed] (2,-3) -- (2,3);

\path[name path = A] (3,-3) -- (3,3);

\filldraw (2,0) circle (0.7pt) node[below right, font=\tiny] {$s=\delta$};

\draw[name path = B, scale=1, domain=-3:3, smooth, variable=\x] plot ({2 - 0.35*\x*\x*\x*\x/(\x*\x*\x*\x + 0.3)}, {\x});

\draw[dashed] (2-0.4,-3) -- (2-0.4,3);

\node [font=\tiny] at (2-0.55,3.1) {$\mathrm{Re}(s)=\delta-\eta_\infty$};

\node [font=\tiny] at (2+0.5,2.7) {$\mathrm{Re}(s)=\delta$};

\tikzfillbetween[of=A and B]{gray,opacity=0.2};


\node [font=\tiny] at (0.3,0.3) {$W$};
\end{tikzpicture}
\caption{Gray area: universal resonance-free region $\mathcal{U}$ of Theorem \ref{thm:dist_zero_free}}
\end{figure}

To gain a sense of Theorem \ref{thm:dist_zero_free} observe that any sequence $(s_n)\subset \CC \setminus \mathcal{U}$, with $\mathcal{U}$ as in \eqref{region_W}, satisfies
$$
\lim_{n\to \infty} \mathrm{Re}(s_n) = \delta \Longrightarrow \lim_{n\to \infty} s_n = \delta.
$$
This immediately implies the following:
\begin{cor}
Let $(X_n)_{n\in \mathbb{N}}$ be a family of finite covers of $X$ with $X$ is as in Theorem \ref{thm:dist_zero_free}. For each $n \in \mathbb{N}$ pick a resonance $s_n$ of $X_n$. If the sequence $(s_n)$ converges to the vertical line $\mathrm{Re}(s)=\delta$ as $n\to \infty$, then $(s_n)$ converges to $\delta$ as $n\to \infty$.
\end{cor}

It is important to emphasize that Theorem \ref{thm:dist_zero_free} is not just an isolated part of this paper. This theorem (or rather a quantitative version thereof, see Corollary \ref{cor:dist_zero_free}) will be crucial to establish our main Theorem \ref{main_theorem}.

We should also remark that the constants in Theorems \ref{main_theorem} and \ref{thm:dist_zero_free} are hardly explicit, due to their reliance on the thermodynamic formalism for twisted Selberg zeta functions. To compute these constants effectively, one would require effective versions of two key ingredients from thermodynamic formalism used in this paper: Proposition \ref{Ruelle_Perron_Frobenius} (Ruelle--Perron--Frobenius Theorem) and Proposition \ref{prop:MN_input} (uniform Dolgopyat estimate). Unfortunately, the author is not aware of effective versions of these results.

\subsection{Congruence subgroups}
Theorem \ref{main_theorem} may be used to establish a uniform resonance-free strip for the family of congruence covers. When $\Gamma$ is a subgroup in $\mathrm{SL}_2(\mathbb{Z})$ and $q\in \mathbb{N}$, the (principal) congruence subgroup of $\Gamma$ of level $q$ is defined by
\begin{equation}\label{defiCongrunceGroup}
\Gamma(q) \defeq \left\{ \text{$\gamma\in \Gamma$ : $\gamma\equiv I$ mod $q$} \right\}.
\end{equation}
Equivalently, $\Gamma(q) $ is the kernel of the reduction map modulo $q$,
\begin{equation*}
\pi_q \colon \Gamma \to \mathrm{SL}_2(\mathbb{Z}/q\mathbb{Z}), \quad \gamma\mapsto \text{$\gamma$ mod $q$.}
\end{equation*}
We denote the associated congruence cover by $X(q)= \Gamma(q)\backslash \mathbb{H}^2.$ When $X$ is convex cocompact, the existence of a uniform spectral gap for $X(q)$ was first proven by Oh--Winter \cite{OhWinter} for $q$ running through the positive square-free integers, and later by Bourgain--Kontorovich--Magee \cite{BKM} for $q$ running over all positive integers with no small prime divisors. 

In both papers \cite{OhWinter,BKM}, one of the main ingredients is the deep result of Bourgain--Varj\'{u} \cite{Bourgain_Varju}, which implies that there is some $\epsilon = \epsilon(\Gamma) > 0$ with following property: for every $q\in \mathbb{N},$ the Cayley graph
$$
\mathcal{G}_q = \mathrm{Cay}(G_q, \pi_q(S)), \quad G_q = \pi_{q}(\Gamma),
$$
is an $\epsilon$-expander. Moreover, there exists an integer $q_0 = q_0(\Gamma)$ such that for all $q\in \mathbb{N}$ co-prime to $q_0$ we have $G_q = \mathrm{SL}_2(\mathbb{Z}/q\mathbb{Z}).$ Combining this result with our main Theorem \ref{main_theorem} \textit{immediately} yields the following:

\begin{theorem}[Resonance-free strip for $X(q)$]\label{congruence_res}
Let $\Gamma \subset \mathrm{SL}_2(\mathbb{Z})$ be a non-elementary, convex cocompact group. Then there exists $\eta>0$ and $q_0\in \mathbb{N}$, depending only on $\Gamma$, such that for all $q\in \mathbb{N}$ with $(q,q_0)=1$ the congruence covers $X(q) = \Gamma(q)\backslash\mathbb{H}^2$ have no new resonances in the half-plane $\mathrm{Re}(s)\geqslant \delta - \eta.$
\end{theorem}

Special cases of Theorem \ref{congruence_res} with explicit spectral gaps have been known for a long time. Selberg's famous $\frac{3}{16}$-theorem \cite{Selberg_3/16} says that for the modular group $\Gamma=\mathrm{SL}_2(\mathbb{Z})$ (in which case $\delta = 1$) we can take $\eta = \frac{1}{4}$. This is equivalent with the first non-zero eigenvalue of $X(q)$ being bounded from below by $\lambda_1(X(q)) \geqslant \frac{3}{16}$ for any $q\in \mathbb{N}$. Selberg himself conjectured that the same should hold true with $\eta = \frac{1}{2}$ (which is equivalent with saying that $\lambda_1(X(q)) \geqslant \frac{1}{4}$). Although this conjecture remains open to date, several quantitative improvements have been made \cite{Iwaniec1990,LRS,Iwaniec96,KimShahidi,HKSar}. We refer to the expository articles of Sarnak \cite{Sar95,Sar05}. The first spectral gap in the infinite-area setting was established by Gamburd \cite{Gamburd1}, who proved that for all finitely generated $\Gamma \subset \mathrm{SL}_2(\mathbb{Z})$ with $\delta>\frac{5}{6}$ and for any sufficiently large prime $q$ one can take $\eta = \delta - \frac{5}{6}$ in Theorem \ref{congruence_res}.

Apart from the intrinsic interest of this problem, finding explicit spectral gaps for congruence covers has many applications to concrete arithmetic questions, see for instance \cite{BGS,Kon09, BK10, KO12, HK15, Ehr19}. Selberg's conjecture (when extended to the infinite-area setting) suggests that the optimal gap should be $\eta=\delta/2$, but this seems out of reach given the current state of the art. Good bounds for $\eta$ in Theorem \ref{congruence_res} are difficult to obtain, mainly due to the reliance on the expander theory developed by Bourgain--Varj\'{u} \cite{Bourgain_Varju} and Bourgain--Gamburd \cite{Bourgain_Gamburd_2} among others. These papers do not provide explicit bounds for the expansion coefficients of the Cayley graphs $\mathcal{G}_q = \mathrm{Cay}(G_q, \pi_q(S)).$ It should be noted that Kowalski \cite{Kowalski_explicit_2012} has obtained an explicit bound for the spectral gap of the Laplacian on $\mathcal{G}_q$ when $q$ is prime. However, according to Kowalski himself, these bounds are far from being sharp. 

Another source of non-explicitness in Theorem \ref{congruence_res} is the fact that our proof of Theorem \ref{main_theorem} (from which Theorem \ref{congruence_res} follows) requires Theorem \ref{thm:dist_zero_free}, which in turn relies on ingredients from thermodynamic formalism with hardly explicit constants.

It is worth noting that the proofs of Oh--Winter \cite{OhWinter} and Bourgain--Kontorovich--Magee \cite{BKM} rely on two fundamental ingredients. Firstly, for sufficiently large primes $p$, they rely on the fact that the quotient group $\Gamma/\Gamma(p)$ is isomorphic to $\mathrm{SL}_2(\mathbb{Z}/p\mathbb{Z})$. Secondly, they make use the following classical result due to Frobenius: the dimension of any non-trivial representation of $\mathrm{SL}_2(\mathbb{Z}/p\mathbb{Z})$ is at least $\frac{p-1}{2}$, which is large when compared to the group order $\vert \mathrm{SL}_2(\mathbb{Z}/p\mathbb{Z})\vert \asymp p^3$. In the terminology introduced by Gowers \cite{Gow08}, the latter means that $\mathrm{SL}_2(\mathbb{Z}/p\mathbb{Z})$ is a highly \textit{quasirandom} group. Our proof of Theorem \ref{main_theorem} (and hence Theorem \ref{congruence_res}) does not require such specific information on the covering group. In fact, Theorem \ref{main_theorem} does not even require the fundamental group of the coverings to be normal subgroups of $\Gamma$. Nonetheless, our proofs bears certain resemblances to \cite{OhWinter,BKM}, such as the usage of the thermodynamic formalism.

We also point out that Theorem \ref{congruence_res} can be extended (in slightly weaker form) to congruence groups coming from arbitrary number fields. This follows directly from Theorem \ref{main_theorem} and the result of Varj\'{u} \cite{Va2012}, which says that for the ring of integers $\mathcal{O}$ of an arbitrary number field, the Cayley graphs of $\mathrm{SL}_2(\mathcal{O}/\mathcal{I})$, with respect to the projections of $S$, form an expander family when $\mathcal{I}$ ranges over square-free ideals of $\mathcal{O}$. (A careful inspection of \cite{Va2012} reveals that this is actually a family of two-sided expanders.) Thus we have the following:

\begin{theorem}\label{thm:varjuconsequence}
Let $\mathcal{O}$ be the ring of integers in a totally real number field and let $\Gamma \subset \mathrm{SL}_2(\mathcal{O})$ be a convex cocompact Fuchsian group. Given an ideal $\mathfrak{q}\subseteq \mathcal{O}$ let $\Gamma(\mathfrak{q})$ be the principal congruence subgroup defined analogously to \eqref{defiCongrunceGroup}. Then there exists a constant $\eta = \eta(\Gamma)>0$ such that $X(\mathfrak{q}) = \Gamma(\mathfrak{q})\backslash \mathbb{H}^2$ has no new resonances in $\mathrm{Re}(s)\geqslant \delta-\eta$, where $\mathfrak{q}\subseteq \mathcal{O}$ ranges over square-free ideals of $\mathcal{O}$ with no small prime factors.
\end{theorem}

The restriction to square-free ideals comes from the main theorem in \cite{Va2012}. The author is not aware whether this restriction can be removed or not. 

At this point the reader may wonder if convex cocompact Fuchsian groups $\Gamma \subset \mathrm{SL}_2(\mathcal{O})$ where $\mathcal{O}$ is an in Theorem \ref{thm:varjuconsequence} actually exist (because if not, Theorem \ref{thm:varjuconsequence} is just a vacuous statement). Fortunately, such groups are actually ubiquitous, as observed in \cite{BKM}: for any pair of hyperbolic elements $\gamma_1, \gamma_2 \in \mathrm{SL}_2(\mathcal{O})$ with no common fixed points on $\overline{\mathbb{R}}$,  $\gamma_1^k, \gamma_2^k$ forms a pair of Schottky generators for all $k$ sufficiently large. In particular, $\Gamma = \langle \gamma_1^k, \gamma_2^k\rangle$ is a convex cocompact subgroup of $\mathrm{SL}_2(\mathcal{O})$.

\subsection{Further explicit constructions}
We now give more examples of families of covers with uniform spectral gap other than those coming from congruence subgroups. Consider the following recipe:

\begin{enumerate}[(1)]
\item Let $G$ be a finite group and let $g_1, \dots, g_m \in G$ be elements that generate $G$.
\item Let $\pi \colon \Gamma \to G$ be the homomorphism defined by $\pi(\gamma_i) = g_i$ for all $i\in \{ 1,\dots, m \}$. Note that $\pi \colon \Gamma \to G$ is surjective since we assume that $g_1, \dots, g_m$ generate the whole group $G$.
\item Finally, let $\Gamma'=\ker(\pi)$ be the kernel of $\pi \colon \Gamma \to G$. By construction, $X'=\Gamma'\backslash\mathbb{H}^2$ is a \textit{$G$-cover} of $X=\Gamma\backslash\mathbb{H}^2$, that is, a finite regular cover whose covering group is isomorphic to $\Gamma'/\Gamma \cong G$. Moreover, the Schreier coset graph associated to this covering is the Cayley graph
$$
\mathcal{G} = \mathrm{Cay}(G, \{ g_1^{\pm}, \dots, g_m^{\pm} \}).
$$
Thanks to Theorem \ref{main_theorem}, if $\mathcal{G}$ is a two-sided $\epsilon$-expander for some $\epsilon>0$, there exists $\eta=\eta(\epsilon,X)>0$ for which $
\mathrm{gap}(X',X)\geqslant \eta$, where $\mathrm{gap}(X',X)$ is the relative spectral gap defined in \eqref{def:relgap}.
\end{enumerate}

Hence, by this recipe, in order to manufacture coverings with uniform spectral gap, it suffices to find families of finite groups $G$ and corresponding generating elements $\{ g_1, \dots, g_m \}\subset G$ such that the resulting Cayley graphs $\mathcal{G}$ are two-sided $\epsilon$-expanders for some uniform $\epsilon > 0.$ It turns out that examples for such families are easy to find in the literature. Let us briefly mention a few of these examples along with their respective implications:

\begin{enumerate}[(1)]
\item By \cite{KLN,BGT2} we know that for every finite simple group $G$ and for every $m\geqslant m_0$ there are elements $g_1,\dots g_m$ in $G$ with the property that the Cayley graph $\mathcal{G}$ is a two-sided $\epsilon$-expander, with $m_0$ and $\epsilon>0$ independent of the group. Note that the number $m$ of Schottky generators for $X$ is related to Euler characteristic by $\chi(X)=1-m$. Hence, by Theorem \ref{main_theorem} it follows that every convex cocompact hyperbolic surface $X$ with negative enough Euler characteristic and for every finite simple group $G$, there exists a $G$-cover $X'$ of $X$ for which $\mathrm{gap}(X',X)\geqslant \eta$, with $\eta > 0$ depending only on $X$.

\item By the main result in \cite{BGGT} (see also references therein) every \textit{simple group of Lie type} $G$ has the property that if we choose elements $g_1, \dots, g_m\in G$ independently and uniformly at random, then, with probability tending to $1$ as $\vert G\vert \to \infty$, these elements generate the whole group $G$ and the Cayley graph $\mathcal{G}$ is a two-sided $\epsilon$-expander for some fixed $\epsilon$. Combining this with Theorem \ref{main_theorem} it follows that for every convex cocompact hyperbolic surface $X$, almost every $G$-cover $X'$ of $X$ has $\mathrm{gap}(X',X)\geqslant \eta$ for some $\eta > 0$ depending only on $X.$ This should be compared with \cite[Corollary 1.1]{NaudMagee}, which is the same statement with ``$G$-cover'' replaced by ``cover of degree $n$'' ($n$ large).

\item By \cite[Theorem 8.3]{BGGT} the statement of the previous item also holds true for direct products $G = G_1 \times\cdots \times G_k$, where $G_1, \dots, G_k$ are simple groups of Lie type.

\item By \cite{BrGa} there exists a density one subset $\mathcal{P}$ of the primes with the following property: for every generating set of $\mathrm{SL}(\mathbb{Z}/p\mathbb{Z})$ with $p\in \mathcal{P},$ the corresponding Cayley graph is a two-sided $\epsilon$-expander where $\epsilon > 0$ is some absolute constant. By Theorem \ref{main_theorem} it follows that for all $p\in \mathcal{P}$, every $\mathrm{SL}(\mathbb{Z}/p\mathbb{Z})$-covering $X'$ of $X$ satisfies $\mathrm{gap}(X',X)\geqslant \eta$ where $\eta > 0$ depends only on $X.$
\end{enumerate}

\subsection{Outline of the proofs}\label{sec:onTheProofs}
Let us now outline the general strategy leading to the proofs of Theorems \ref{main_theorem} and \ref{thm:dist_zero_free}. Throughout, we take advantage of the following two facts. First, the resonances of the surface $X = \Gamma\backslash\mathbb{H}^2$ correspond to the zeros of the corresponding Selberg zeta function, denoted by $Z_{\Gamma}(s)$. Secondly, given a finite cover $X' \to X$ we can assume that the fundamental group $\Gamma'$ of $X'$ is a finite-index subgroup of $\Gamma$. In this case, the resonances of $X'$ occur as zeros of the ``twisted'' Selberg zeta function $Z_\Gamma(s,\lambda_{\Gamma/\Gamma'})$, where $\lambda_{\Gamma/\Gamma'}$ is the induced representation of the identity on $\Gamma'$ to $\Gamma$, see §\ref{sec:VZ}. To make the connection with expanders we notice first that the adjacency operator of the Schreier coset graph can be expressed easily in terms of $\lambda_{\Gamma/\Gamma'}$, see §\ref{sec:Schreier}. Being an expander graph thus becomes a property of the induced representation. Therefore, it is natural to reformulate the problem more generally in terms of the zeros of twisted Selberg zeta functions $Z_\Gamma(s,\rho)$ for general finite-dimensional representations $\rho$ of $\Gamma$. In this view,

\begin{enumerate}[(1)]
\item our main Theorem \ref{main_theorem} is a special case of Theorem \ref{thm:second_part} and 
\item Theorem \ref{thm:dist_zero_free} is a special case of Theorem \ref{thm:dist_zero_free:reformul}.
\end{enumerate}

We also exploit the fact that for Schottky groups $\Gamma$, $Z_\Gamma(s,\rho)$ can be realized as the Fredholm determinant of a family of \textit{transfer operators} $\mathcal{L}_{s,\rho}$, $s\in \CC$, acting on a functional space consisting of $V$-valued functions on a neighbourhood $I\subset \mathbb{R}$ of the limit set $\Lambda$. More details will be given in §\ref{sec:transfer_operator}. Any zero $s=\sigma+it$ of $Z_\Gamma(s,\rho)$ thus corresponds to a $1$-eigenfunction of the operator $\mathcal{L}_{s,\rho}$, i.e., there exists some non-zero $F\colon I\to V$ satisfying
$$
\mathcal{L}_{s,\rho}^N F = F.
$$
for all $N\in \mathbb{N}$. This can be restated as follows: there exists some non-zero $f \colon I \to V$ satisfying
\begin{equation}\label{outline:norm}
\widetilde{\mathcal{L}}_{s,\rho}^N f= e^{-NP(\sigma)}f,
\end{equation}
where $\widetilde{\mathcal{L}}_{s,\rho}$ denotes the \textit{normalized} transfer operator, and $P$ is the \textit{topological pressure} associated to the natural dynamical system induced by the action of $\Gamma$ on the union of Schottky intervals $I$. In §\ref{sec:toppressure}, details pertaining to the thermodynamic formalism, including the topological pressure, will be provided. The normalized transfer operator will be introduced in §\ref{sec:norm_to}.

\subsubsection{Outline of proof of Theorem \ref{thm:dist_zero_free}}
Let us first describe the proof of Theorem \ref{thm:dist_zero_free}. Assume that $s=\sigma + it$ is a resonance for a finite cover of $X'=\Gamma'\backslash \mathbb{H}^2$. Hence, $s=\sigma + it$ must be a zero of $Z_{\Gamma}(s,\rho)$ where $\rho$ is the induced representation of $\Gamma'/\Gamma$ and consequently, there exists a non-zero $f\colon I\to V$ satisfying \eqref{outline:norm}. When $\vert t\vert$ is large we use a high-frequency norm estimate (Proposition \ref{prop:hf_norm_estimate}). This estimate, which follows essentially from Magee--Naud \cite{NaudMagee}, tells us that for all $\vert t\vert \gg 1$ and for $N\approx \log \vert t\vert$
$$
\Vert \widetilde{\mathcal{L}}_{s,\rho}^N f\Vert_{C^1(I)} \leqslant \frac{1}{\vert t\vert} \Vert f\Vert_{C^1(I)}.
$$
Combining this with \eqref{outline:norm} yields the desired contradiction, provided $\vert t\vert$ is large enough and $\sigma$ is sufficiently close to $\delta$. Now we assume alternatively that $\vert t\vert \ll 1$ is bounded from above. In this case we apply the \textit{tensor power trick} (see §\ref{sec:tensorProductTrick}), which can be described informally as follows: for some $k\in \mathbb{N}$, let $h=f^{\otimes k}$ be the $k$-th tensor power of $f$ and let $\nu = \rho^{\otimes k}$ be the $k$-th tensor power representation of $\rho$. Then, letting $x\in I$ be a maximizer of the function $x\mapsto\Vert f(x)\Vert_V$, one can show that
\begin{equation}\label{outline:approxtensor}
\widetilde{\mathcal{L}}_{\sigma + itk,\nu}^N h(x) \approx e^{-NP(\sigma)}h(x).
\end{equation}
Here, the symbol $\approx$ means that the difference of both sides is small measured in the norm of $V^{\otimes k}$, provided $\sigma$ is close to $\delta$ and $k$ is fixed. But now (assuming $t\neq 0$) we take $k$ so that $k\vert t\vert \gg 1$ is large enough. We can once again apply the high-frequency estimate. This allows us to obtain a contradiction to \eqref{outline:approxtensor}, provided $\sigma$ is sufficiently close to $\delta$ in terms of $t$. Pursuing this analysis carefully yields a contradiction for all $\sigma\geqslant \delta-\eta(t)$, where $\eta \colon \mathbb{R}\to \mathbb{R}$ is the function given in Theorem \ref{thm:dist_zero_free}.

\subsubsection{Outline of proof of main theorem}
Let us now turn our attention to Theorem \ref{main_theorem}. Some new notations are needed. The group $\Gamma$ is freely generated by the Schottky generating set $\gamma_1, \dots, \gamma_m$ and their inverses, which we denote by
\begin{equation*}
\text{$\gamma_{m+1} = \gamma_1^{-1}$, $\gamma_{m+2} = \gamma_2^{-1},$ ... , $ \gamma_{2m} = \gamma_{m}^{-1}$.}
\end{equation*}
Thus the elements of $\Gamma$ can be indexed by reduced words $\textbf{a} = a_1 \cdots a_N$ with letters $a_1, \dots , a_N$ in the alphabet $[2m] = \{ 1,\dots,2m \}$. Given such a word we write
\begin{equation*}
\gamma_{\textbf{a}} \defeq \gamma_{a_1}\circ \cdots \circ \gamma_{a_N}.
\end{equation*}
We denote by $W_N^j$ the set of reduced words of length $N$ not ending with the letter $j.$ 

Now fix a subgroup $\Gamma'\subset\Gamma$ and assume that the associated Schreier graph is a two-sided $\epsilon$-expander for some $\epsilon> 0$. Let $s=\sigma+it$ be a \textit{new} resonance for $X' = \Gamma'\backslash \mathbb{H}^2$. Then $s=\sigma+it$ is a zero of $Z_{\Gamma}(s,\rho)$ where $\rho$ is the induced representation of $\Gamma/\Gamma'$ \textit{minus} the identity. Once again, this zero corresponds to a solution $f$ of equation \eqref{outline:norm}. Suppose that $\sigma \geqslant \delta - \eta$. We seek to obtain a contradiction when $\eta>0$ is sufficiently small in terms of $\epsilon$. Let $x\in I$ be a maximizer of $x \mapsto\Vert f(x)\Vert_V$. Using Proposition \ref{prop:strict_convexity}, which is essentially a strict convexity argument, we deduce that there exists some $j\in [2m]$ such that for all $\textbf{a}\in W_N^j$,
\begin{equation}\label{outline:strconv}
f(x) \approx \gamma_{\textbf{a}}'(x)^{it} \rho(\gamma_{\textbf{a}})^{-1}  f(\gamma_{\textbf{a}}(x)).
\end{equation}
Here we write $\approx$ to mean that the difference of both sides (measured in the $V$-norm) is small as $\eta \to 0^+$. Now we invoke Corollary \ref{cor:dist_zero_free} (a quantitative version of Theorem \ref{thm:dist_zero_free}). It says that if $s=\sigma+it$ is a resonance for $X'$ with real part $\sigma$ close enough to $\delta$, then we must have $t \approx 0$. It follows that $\gamma_{\textbf{a}}'(x)^{it} \approx 1 $, so \eqref{outline:strconv} becomes
\begin{equation}\label{outline:approx}
f(x) \approx \rho(\gamma_{\textbf{a}})^{-1}  f(\gamma_{\textbf{a}}(x)).
\end{equation}
Now we take the average of this equation over all words $\textbf{a}\in W_N^j$:
$$
f(x) \approx \mathbb{E}_{\textbf{a}\in W_N^j} \left( \rho(\gamma_{\textbf{a}})^{-1}  f(\gamma_{\textbf{a}}(x)) \right).
$$
Finally, we apply the estimate from Proposition \ref{prop:exp_dec_average}, which is precisely where we appeal to the expansion property. Roughly speaking, this estimate tells us that
$$
\left\Vert \mathbb{E}_{\textbf{a}\in W_N^j} \left( \rho(\gamma_{\textbf{a}})^{-1}  f(\gamma_{\textbf{a}}(x)) \right)\right\Vert \ll e^{-c\epsilon N} \Vert f\Vert_\infty.
$$
Combining this with \eqref{outline:approx} yields
$$
\Vert f\Vert_{\infty} = \Vert f(x)\Vert_V  \ll e^{-c\epsilon N} \Vert f\Vert_\infty
$$
and we now obtain the desired contradiction once $N$ is large enough.

It should be noted that the actual proofs deviate considerably from the description given here. Let us mention two aspects omitted from this exposition. First, we clearly need to replace the ambiguous symbol $\approx$ by actual estimates. Second, we need to control the size of the derivative $f'$ in terms of the size of $f$. This is dealt with in Lemma \ref{estimate_der_eigenfct}.

\subsection{Structure of the paper} The rest of this paper is structured as follows.
In §\ref{sec:preliminaries} we provide some background on twisted Selberg zeta functions, transfer operators, the Venkov--Zograf induction formula, Schreier graphs, and Schottky groups.

In §\ref{sec:dist_zero_free} we prove Theorem \ref{thm:dist_zero_free:reformul}, the generalized version of Theorem \ref{thm:dist_zero_free}.

In §\ref{sec:exploiting_expander} we exploit the expansion property to derive estimates for certain twisted sums over reduced words. The main estimate here is Proposition \ref{prop:exp_dec_average}.

In §\ref{sec:part2mainthm} we proof Theorem \ref{thm:second_part}, the generalized version of Theorem \ref{main_theorem}.

\subsection{Notation}\label{sec:notation}

\begin{itemize}
\item We write $f(x)\ll g(x)$ or $ f(x) = O(g(x)) $ interchangeably to mean that there exists an implied constant $C>0$ for which $\vert f(x)\vert \leqslant C \vert g(x)\vert$ for all $x\geqslant C.$ We write $C=C(a,b,\dots)$ to emphasize that $C$ depends on $a, b, \dots$.

\item We write $s=\sigma + it$ to mean that $s$ is a complex number with real part equal to $\sigma$ and imaginary part equal to $t$.

\item Given a finite set $A$ we denote by $\vert A\vert$ its cardinality and we write $\mathbb{E}_{a\in A} f(a)$ for the average
$$
\mathbb{E}_{a\in A} f(a) \defeq \frac{1}{\vert A\vert} \sum_{a\in A} f(a).
$$

\item Given $N\in \mathbb{N}$ we write $[N] = \{ 1,\dots, N \}.$
\end{itemize}

\subsection{Acknowledgments} I would like to thank Michael Magee for bringing to my attention his joint work with Jean Bourgain and Alex Kontorovich \cite{BKM}.

\section{Preliminaries}\label{sec:preliminaries}

\subsection{Hyperbolic geometry}\label{sec:hypgeom} We start by reviewing some basic facts about hyperbolic geometry, referring the reader to Borthwick's book \cite{Borthwick_book} for in-depth account of the material presented here. One of the standard models for the hyperbolic plane is the Poincar\'{e} half-plane
$$
\mathbb{H}^2=\{ x+iy\in \C\ :\ y>0\},
$$
endowed with its standard metric of constant curvature $-1$,
$$
ds^2=\frac{dx^2+dy^2}{y^2}.
$$ 
The group of orientation-preserving isometries of $(\mathbb{H}^2, ds)$ is isomorphic to $\mathrm{PSL}_2(\R)$, which acts on $\mathbb{H}^2$ by M\"{o}bius transformations
$$
\gamma=\pmat{a}{b}{c}{d}\in \mathrm{PSL}_2(\R),\quad z\in\mathbb{H}^2 \Longrightarrow  \gamma(z) = \frac{az+b}{cz+d}.
$$
This action extends naturally to the extended complex plane $\overline{\C} = \C \cup \{ \infty\}$. An element $\gamma\in \mathrm{PSL}_2(\R)$ is called 
\begin{itemize}
\item ``hyperbolic'' if $\vert \tr(\gamma)\vert >2$, which implies that $\gamma$ has two distinct fixed points on the boundary $\partial \mathbb{H}^2$,
\item ``parabolic'' if $\vert \tr(\gamma)\vert =2$, which implies that $\gamma$ has precisely one fixed point on $\partial \mathbb{H}^2$, and
\item ``elliptic'' if $\vert \tr(\gamma)\vert < 2$, which implies that $\gamma$ has precisely one fixed point in $\mathbb{H}^2$.
\end{itemize}

Discrete subgroups of $\mathrm{PSL}_2(\R)$ are called ``Fuchsian'' groups. A Fuchsian group is ``torsion-free'' if it contains no elliptic elements. It is called ``convex cocompact'' if it is finitely generated and if it contains neither elliptic nor parabolic elements. This is equivalent to the ``convex core'' of $X=\Gamma\backslash\mathbb{H}^2$ being compact. Every infinite-area, convex cocompact surface is isometric to a quotient $\Gamma\backslash \mathbb{H}^2$ with $\Gamma$ a so-called ``Schottky group'', which we will define in §\ref{Schottky_groups} below.

\subsection{Twisted Selberg zeta functions}
Given a finitely generated Fuchsian group $\Gamma<\mathrm{SL}_2(\R)$, the set of prime periodic geodesics on $X=\Gamma\backslash\mathbb{H}^2$ is bijective to the set $[\Gamma]_{\mathsf{prim}}$ of $\Gamma$-conjugacy classes of the primitive hyperbolic elements in $\Gamma$ (see for instance \cite[Proposition 2.25]{Borthwick_book}). Let $\ell(\gamma)$ denote the length of the geodesic corresponding to the conjugacy class $[\gamma] \in [\Gamma]_{\mathsf{prim}}$.

The Selberg zeta function is defined for $\mathrm{Re}(s) > \delta$ by the infinite product
\begin{equation}\label{selbergZetaDefi}
Z_{\Gamma}(s)\defeq\prod_{k=0}^\infty \prod_{[\gamma] \in [\Gamma]_{\mathsf{prim}}}\left( 1 - e^{-(s+k)\ell(\gamma)}\right).
\end{equation}
It admits a meromorphic continuation to $s\in \mathbb{C}.$ By Patterson--Perry \cite{Patt_Perry} the zero set of $Z_\Gamma(s)$ consists of so-called ``topological zeros'' at $s= -k$ for $k\in \mathbb{N}_0$, and the set of resonances, repeated according to multiplicity. Therefore, any problem about resonances can be rephrased as a question about the distribution of the zeros of the Selberg zeta function.

In this paper, we are interested in ``twisted'' Selberg zeta functions. They are defined for any finite-dimensional, unitary representation $\rho\colon \Gamma\to\mathrm{U}(V)$ by
\begin{equation}\label{def_szf_twisted}
Z_{\Gamma}(s,\rho)\defeq\prod_{k=0}^\infty \prod_{[\gamma] \in [\Gamma]_{\mathsf{prim}}}\det\left( I_V - \rho(\gamma)e^{-(s+k)\ell(\gamma)}\right).
\end{equation}
Note that \eqref{def_szf_twisted} reduces to the classical Selberg zeta function \eqref{selbergZetaDefi} when $\rho$ is taken to be the trivial one-dimensional representation of $\Gamma.$ Moreover, it follows directly from this product definition that we can factorize
\begin{equation}\label{decompFormulaZeta}
Z_\Gamma(s,\rho_1 \oplus \rho_2) = Z_\Gamma(s,\rho_1)Z_\Gamma(s,\rho_2),
\end{equation}
where $\rho_1 \oplus \rho_2$ denotes the orthogonal direct sum of $\rho_1$ and $\rho_2$.

\subsection{Finite covers and Venkov--Zograf formula}\label{sec:VZ} 
Let $\Gamma'$ be a finite-index subgroup of $\Gamma$ and let $X'= \Gamma'\backslash \mathbb{H}^2$ be the corresponding cover of $X$. Let
\begin{equation}
\lambda_{\Gamma/\Gamma'} \defeq \mathrm{Ind}_{\Gamma'}^{\Gamma}(\textbf{1}_{\Gamma'})
\end{equation}
be the induced representation of the trivial one-dimensional representation $\textbf{1}_{\Gamma'}$ on $\Gamma'$ to the larger group $\Gamma.$ The \textit{Venkov--Zograf (induction) formula} asserts that
\begin{equation}\label{VZ_ind_formula}
Z_{\Gamma'}(s) = Z_\Gamma(s,\lambda_{\Gamma/\Gamma'}).
\end{equation}
This was proven by Venkov--Zograf \cite{VenkovZograf} in the case where $\Gamma$ is a cofinite Fuchsian group (see also \cite{Venkov_book}). For an extension of this formula to the non-cofinite case we refer to \cite{FP_szf}.

Let  $n=[\Gamma:\Gamma']$ be the index of $\Gamma'$ in $\Gamma$ and let $g_1, \dots, g_n$ be a full set of representatives in $\Gamma$ of the left cosets in $\Gamma/\Gamma'$. The induced representation can be thought of as acting on the space
\begin{equation}\label{reprspaceind}
V_{\Gamma/\Gamma'} \defeq \mathrm{span}_\mathbb{C} \{ g_1, \dots, g_n \} = \left\{ \sum_{i=1}^n \alpha_i g_i : \alpha_1, \dots, \alpha_n\in \mathbb{C} \right\},
\end{equation}
To describe the action, note that for every $\gamma\in \Gamma$ and for each $i\in [n]$ there exists $\sigma(i)\in [n]$ and $\tilde{\gamma}\in \Gamma'$ such that $ \gamma g_i = g_{\sigma(i)} \tilde{\gamma}. $ The action of $\lambda_{\Gamma/\Gamma'}$ is then given by
$$
\lambda_{\Gamma/\Gamma'}(\gamma)\left( \sum_{i=1}^n \alpha_i g_i \right) = \sum_{i=1}^n \alpha_i g_{\sigma(i)}.
$$
In fact, $\sigma\in S_n$ is a permutation of $[n]$. With respect to the basis $g_1, \dots, g_n$, the element $\lambda_{\Gamma/\Gamma'}(\gamma)$ acts on $V_{\Gamma/\Gamma'}$ by the permutation matrix associated to $\sigma.$ Moreover, the induced representation splits as an orthogonal direct sum
$$
\lambda_{\Gamma/\Gamma'} = \textbf{1}_\Gamma  \oplus \lambda_{\Gamma/\Gamma'}^0,
$$
where $\lambda_{\Gamma/\Gamma'}^0$ is representation acting on the $(n-1)$-dimensional subspace given by
\begin{equation}\label{reprspaceind0}
V_{\Gamma/\Gamma'}^0 \defeq \left\{ \sum_{i=1}^n \alpha_i g_i \in V :  \sum_{i=1}^n \alpha_i = 0  \right\}.
\end{equation}
From \eqref{decompFormulaZeta} we deduce that the Selberg zeta function of the subgroup $\Gamma'\subset \Gamma$ can be written as
\begin{equation}\label{VZ_ind_formula_2}
Z_{\Gamma'}(s) = Z_\Gamma(s) Z_\Gamma(s,\lambda_{\Gamma/\Gamma'}^0).
\end{equation}
In light of this factorization we obtain the following crucial fact: 
\begin{center}
\textit{Every new resonance for $X' = \Gamma'\backslash\mathbb{H}^2$ with respect to $X=\Gamma\backslash\mathbb{H}^2$ is a zero of the twisted Selberg zeta function $Z_\Gamma(s,\lambda_{\Gamma/\Gamma'}^0).$}
\end{center}

\subsection{Schreier graphs and Property $(\tau)$}\label{sec:Schreier}
A graph is a tuple $\mathcal{G}=(\mathcal{V},\mathcal{E})$ where $\mathcal{V}$ is a set, called the vertex set of $\mathcal{G}$, and $\mathcal{E}$ is a collection of unordered pairs $\{ a,b \}$ of distinct elements $a,b\in \mathcal{V}$, called the edge set of $\mathcal{G}$. For technical reasons, we allow $\mathcal{E}$ to contain repeated edges (also called parallel edges). The graph $\mathcal{G}$ is called ``$k$-regular'' if each vertex is contained in exactly $k$ edges, where edges are counted with repetitions. The adjacency operator $A_\mathcal{G}\colon L^2(\mathcal{V})\to L^2(\mathcal{V})$ is defined for all $f\in L^2(\mathcal{V})$ and $a\in \mathcal{V}$ as 
$$
(A_\mathcal{G}f)(a) \defeq \sum_{\substack{b\in \mathcal{V} \\ \{ a,b \}\in \mathcal{E}}} f(b).
$$
Suppose $\mathcal{G}$ is $k$-regular. The operator $A_\mathcal{G}$ is self-adjoint on $L^2(\mathcal{V})$ and  its operator norm is bounded from above by $k$. Hence, by the spectral theorem, all the eigenvalues of $A_\mathcal{G}$ lie in the interval $[-k,k].$ Moreover, the largest eigenvalue of $A_\mathcal{G}$ equals $k$ as $A_\mathcal{G}\textbf{1} = k \textbf{1}$, where $\textbf{1}$ is the constant one function.

Now recall that given a group $\Gamma$, a finite-index subgroup $\Gamma'\subset \Gamma$, and a generating set $S = \{ \gamma_1, \dots, \gamma_m \}$ of $\Gamma$, the associated Schreier coset graph $\mathcal{G} = \mathcal{G}(\Gamma,\Gamma',S)$ is constructed in the following way: the vertices are the left cosets in $\Gamma/\Gamma'$ and the edges are given by $\{  x\Gamma', s x\Gamma'\}$ with $s\in S\cup S^{-1}.$ Note that this graph has repeated edges if and only if $s_1 x\Gamma' = s_2 x \Gamma'$ for some $x\in \Gamma$ and for two distinct elements $s_1,s_2\in S\cup S^{-1}$. By construction, $\mathcal{G}$ is a $2m$-regular graph where $m = \vert S\vert$ with adjacency operator
$$
(A_\mathcal{G}f)(x\Gamma') \defeq \sum_{j=1}^{m} \left(  f(\gamma_j x\Gamma')+ f(\gamma_j^{-1}x\Gamma') \right).
$$
Recall from the introduction that $\mathcal{G}$ is called
\begin{itemize}
\item an ``$\epsilon$-expander'' if all the eigenvalues of $A_{\mathcal{G}}$, other than the largest eigenvalue $2m$, lie in $[-2m, 2m(1-\epsilon)]$, and
\item a ``two-sided $\epsilon$-expander'' if all the eigenvalues of $A_{\mathcal{G}}$, other than $2m$, lie in the interval $[-2m(1-\epsilon), 2m(1-\epsilon)]$.
\end{itemize}

We now define the following operator for any representation $(\rho,V)$ of $\Gamma$:
\begin{equation}
T(\rho) \defeq \frac{1}{2m} \sum_{j=1}^{m} \left( \rho(\gamma_j) + \rho(\gamma_j)^{-1} \right) \in \mathrm{End}(V).
\end{equation}
We will henceforth assume that $\rho$ is unitary, which is to say that $\rho(\gamma)^{\ast} = \rho(\gamma)^{-1}$ for all $\gamma\in \Gamma$. It then follows that $T(\rho)$ is self-adjoint, that is, $T(\rho)^\ast = T(\rho)$; moreover, since $\Vert T(\rho)\Vert \leqslant 1$, all the eigenvalues of $T(\rho)$ lie inside $[-1,1].$

Let $\lambda_{\Gamma/\Gamma'}$ be the induced representation described in \ref{sec:VZ}. Observe that $ T(\lambda_{\Gamma/\Gamma'})$ is equal to the normalized adjacency matrix of the Schreier coset graph $\mathcal{G}  = \mathcal{G}(\Gamma,\Gamma',S) $, i.e., if we identify the representation space $V_{\Gamma/\Gamma'}$ with $L^2(\Gamma/\Gamma')$ we have $T(\lambda_{\Gamma/\Gamma'}) =  \frac{1}{2m} A_\mathcal{G} $. As a consequence we obtain

\begin{lemma}\label{lem:rel_trho_exp}
Let $\mathcal{G} = \mathcal{G}(\Gamma,\Gamma',S) $ be the Schreier graph as above. Then the following hold true:
\begin{enumerate}[(1)]
\item The graph $\mathcal{G}$ is an $\epsilon$-expander if and only if all the eigenvalues of $T(\lambda_{\Gamma/\Gamma'}^0)$ are contained inside the interval $[-1,1-\epsilon]$, or equivalently, if
$$
\inf_{v\in V, \Vert v \Vert = 1} \langle (I - T(\lambda_{\Gamma/\Gamma'}^0))v, v\rangle_V\geqslant \epsilon.
$$

\item The graph $\mathcal{G}$ a two-sided $\epsilon$-expander if and only if all the eigenvalues of $T(\lambda_{\Gamma/\Gamma'}^0)$ are contained in $[-1+\epsilon,1-\epsilon].$
\end{enumerate}
\end{lemma}

We define the \textit{Kazhdan distance} between a unitary representation $(\rho,V)$ of $\Gamma$ and the identity $\textbf{1}_\Gamma$ by
\begin{equation}\label{defi:kazhdan_dist}
\kappa_S(\rho, \textbf{1}_\Gamma) \defeq \inf_{\substack{ v\in V \\ \Vert v\Vert=1 }} \max_{\gamma\in S \cup S^{-1}}\Vert \rho(\gamma)v-v\Vert.
\end{equation}
We say that $\Gamma$ has \textit{Property ($\tau$)} with respect to a family of finite-index subgroups $(\Gamma_n)_{n\in I}$ of $\Gamma$ if there exists some $\epsilon>0$ such that $\kappa_S(\lambda_{\Gamma/\Gamma_n}^0, \textbf{1}_\Gamma)\geqslant \epsilon$ for all $n\in I.$ Property ($\tau$) does not depend on the choice of a set of generators. In fact, one can show easily that if $S$ and $S'$ are finite sets of generators of $\Gamma$, there are positive constants $c_1$ and $c_2$ such that
$$
c_1 \kappa_S(\rho, \textbf{1}_\Gamma) \leqslant \kappa_{S'}(\rho, \textbf{1}_\Gamma) \leqslant c_2 \kappa_S(\rho, \textbf{1}_\Gamma).
$$ 
Property $(\tau)$ is an important group-theoretic concept and was first defined by Lubotzky and Zimmer \cite{LubotzkyZimmer}. It can be seen as a weaker variant of Kazhdan's Property (T), introduced in \cite{KazhdanT}, and is exactly what is needed to make the Cayley graphs $\mathrm{Cay}(\Gamma/\Gamma_n, S)$ expanders:

\begin{lemma}\label{lem:tauexp} The following conditions are equivalent:
\begin{enumerate}[(i)]
\item $\Gamma$ has Property $(\tau)$ with respect to $(\Gamma_n)_{n\in I}$
\item For any generating subset $S$ of $\Gamma$, the Cayley graphs $\mathrm{Cay}(\Gamma/\Gamma_n, S)$ form a family of expanders.
\end{enumerate}
\end{lemma}

\begin{proof}
A simple calculation shows that for every unitary representation $(\rho,V)$ of $\Gamma$ and for every $v\in V$ we have
\begin{equation}\label{ineq:CSandcalc}
\frac{1}{4m}\sum_{\gamma\in S \cup S^{-1}} \Vert v -\rho(\gamma)v\Vert^2 = \langle (I-T(\rho))v,v\rangle_V.
\end{equation}
It follows that
$$
\frac{1}{4m}  \kappa_S(\rho, \textbf{1}_\Gamma)^2 \leqslant\inf_{v\in V, \Vert v \Vert = 1}\langle (I-T(\rho))v,v\rangle_V\leqslant \frac{1}{2}  \kappa_S(\rho, \textbf{1}_\Gamma)^2.
$$
Lemma \ref{lem:tauexp} now follows directly from Lemma \ref{lem:rel_trho_exp}.
\end{proof}

\subsection{Schottky groups}\label{Schottky_groups}
We now define Schottky groups, referring the reader to \cite[Chapter 15]{Borthwick_book} for a comprehensive discussion. A Schottky group is a convex cocompact subgroup $\Gamma \subset \mathrm{SL}_2(\R)$ constructed in the following way:
\begin{itemize}
\item Fix a positive integer $m$ and open, non-intersecting euclidean disks $D_1, \dots, D_{2m} \subset \mathbb{C}$ (in no particular order) with centers on the real line.
\item For every $j\in \{ 1,\dots,m \}$ let $\gamma_j\in \mathrm{SL}_2(\mathbb{R})$ be the isometry that maps the exterior of $D_j$ to the interior of $D_{j+m}$. Moreover, for every $j\in \{ m+1,\dots,2m \}$ put
$$
\gamma_{j} \defeq \gamma_{j-m}^{-1}.
$$
For notational purposes it is helpful to the define the indices cyclically modulo $ 2m $ so that
\begin{equation*}
\text{$\gamma_{j+m} = \gamma_{j}^{-1}$ and $\gamma_{j+2m} = \gamma_{j}.$}
\end{equation*}
These definitions imply that for all $i\neq j$ we have 
$$
\overline{\gamma_{i}(D_j)} \subset D_{i+m}.
$$
\item Let $\Gamma \subset \mathrm{SL}_2(\R)$ be the free group generated by the elements $\gamma_1, \dots, \gamma_{2m}.$
\end{itemize}

By Button's result \cite{Button} every convex cocompact hyperbolic surface $X$ can be realized as the quotient of $\mathbb{H}^2$ by a Schottky group $\Gamma$, see also \cite[Theorem~15.3]{Borthwick_book}. Moreover, the complement $\mathbb{H}^2 \smallsetminus \bigsqcup_{j=1}^{2m} D_j $ provides a fundamental domain for the action of $\Gamma$ on $\mathbb{H}^2$.
\begin{figure}[h]
\centering
\captionsetup{justification=centering}
\begin{tikzpicture}[xscale=1.5, yscale=1.5]
\draw (-4,0) -- (4,0) node [right, font=\small]  {$  \partial \mathbb{H}^2 $};
\draw (0,1) node [above, font=\small]  {$ \mathbb{H}^2 $};
\draw (-3,0) circle [radius = 0.4];
\draw (-1.3,0) circle [radius = 0.8];
\draw (-0.1,0) circle [radius = 0.3];
\draw (1,0) circle [radius = 0.5];
\draw (2.3,0) circle [radius = 0.4];
\draw (3.5,0) circle [radius = 0.2];
\draw (-3,-0.4) node[below,font=\small] {$ D_{1} $};
\draw (-1.3,-0.8) node[below,font=\small] {$ D_{4} $};
\draw (-0.1,-0.3) node[below,font=\small] {$ D_{2} $};
\draw (1,-0.5) node[below,font=\small] {$ D_{3} $};
\draw (2.3,-0.4) node[below,font=\small] {$ D_{5} $};
\draw (3.5,-0.2) node[below,font=\small] {$ D_{6} $};
 \draw [arrow, bend left]  (-3,0.4) to (-1.3,0.8);
 \draw [arrow, bend left]  (-0.1,0.3) to (2.3,0.4);
 \draw [arrow, bend left]  (1,0.5) to (3.5,0.2);
 \draw (-2.2,0.9) node[above,font=\small] {$ \gamma_{1} $};
 \draw (1.1,0.7) node[above,font=\small] {$ \gamma_{2} $};
 \draw (2.4,0.7) node[above,font=\small] {$ \gamma_{3} $};
\end{tikzpicture}
\caption{A configuration of Schottky disks and isometries with $ m=3 $}
\end{figure}

A Fuchsian group $\Gamma$ is called ``elementary'' if its limit set $\Lambda$ is finite. If $\Gamma$ is a Schottky group as above, then it is elementary if and only if $m=1$. In this case, $\Gamma = \langle \gamma\rangle$ is a cyclic group generated by a single hyperbolic element $\gamma$ and $\Gamma\backslash \mathbb{H}^2$ is a hyperbolic cylinder.\\

\textit{Throughout the rest of this paper we assume that $X=\Gamma\backslash \mathbb{H}^2$ is a convex cocompact quotient where $\Gamma$ is a non-elementary Schottky group with Schottky data $D_1, \dots, D_{2m}$ and $\gamma_1, \dots, \gamma_{2m}$ as above. Moreover, we set $D\defeq  \bigsqcup_{j=1}^{2m} D_j$, $I_j = D_j\cap \mathbb{R}$, and $I = I_1 \sqcup \cdots \sqcup I_{2m}.$ These notations and assumptions will remain fixed in the sequel.}

\subsection{Transfer operator}\label{sec:transfer_operator} 
In what follows, let $V$ be a finite-dimensional complex vector space with hermitian inner product $\langle\cdot,\cdot\rangle_V$ and induced norm $\Vert v\Vert_V = \sqrt{\langle v,v\rangle_V}.$ Let $\rho\colon\Gamma\to\mathrm{U}(V)$ be a unitary representation. Recall that  ``unitary'' means that for all $\gamma\in \Gamma$ and for all $ v,w\in V$ we have $\langle\rho(\gamma)v,\rho(\gamma)w\rangle_V = \langle v,w\rangle_V$ and in particular $\Vert \rho(\gamma)v\Vert_V = \Vert v\Vert_V$. Let $H^2(D,V)$ be the Hilbert space of $V$-valued, square-integrable, holomorphic functions on $D =  \bigsqcup_{j=1}^{2m} D_j$:
\begin{equation}\label{defi_bergman}
H^2(D,V) \defeq  \left\{ \text{$F\colon D\to V$ holomorphic} \left\vert \; \left\Vert F \right\Vert < \infty \right.\right\}
\end{equation}
with $L^2$-norm given by
$$
\Vert F\Vert^2 \defeq \int_D \Vert F(z)\Vert_V^2 \dvol(z).
$$
Here, $\vol$ denotes the Lebesgue measure on the complex plane and $\Vert \cdot \Vert_V$ is the norm of the representation space $V$. We now define the holomorphic family of operators, parametrized by $s\in \mathbb{C}$,
\begin{equation}\label{op}
\mathcal{L}_{s,\rho}\colon H^2(D,V) \to H^2(D,V)
\end{equation}
by the formula
\begin{equation}\label{transf_defi}
\mathcal{L}_{s,\rho} F(z) \defeq \sum_{\substack{ i=1 \\ i\neq j }}^{2m} \gamma_i'(z)^s \rho(\gamma_i)^{-1} F(\gamma_i (z))\ \mathrm{if}\ z\in D_j
\end{equation}
for all $F\in H^2(D,V)$. Given $\gamma\in \Gamma$ and $z\in D$, the composition $\rho( \gamma )^{-1}F(z)$ is the result of applying the unitary endomorphism $\rho(\gamma)^{-1}\in \mathrm{U}(V)$ to the vector $F(z) \in V$.

In the one-dimensional case $V=\C$, the functional space \eqref{defi_bergman} reduces to the classical Bergman space $H^2(D)$ and the operator \eqref{op} reduces to the classical transfer operator for Schottky groups which can be found for instance in Borthwick's book \cite[Chapter 15]{Borthwick_book}. The next result relates the transfer operator to the Selberg zeta function:

\begin{prop}[Fredholm determinant identity]\label{Fredholm_identity} For every $s\in \C$ the operator \eqref{op} is trace class and we have the identity 
\begin{equation}\label{Fredholm_identity_equation}
Z_{\Gamma}(s,\rho) = \det(1-\mathcal{L}_{s,\rho}).
\end{equation}
\end{prop}

For the trivial representation $\rho = \textbf{1}_\Gamma$, identities in the spirit of \eqref{Fredholm_identity_equation} are well-known in \textit{thermodynamic formalism}, a subject going back to Ruelle \cite{Ruelle_zeta}. The relation between the Selberg zeta function and transfer operators has been studied by a number of different authors. For the convex cocompact setting (no cusps) we refer to \cite{Pollicott, Pollicott_Rocha, Guillope_Lin_Zworski}. The extension to non-trivial twists $\rho$ can be found in the more recent papers \cite{FP_szf,Pohl_Soares,Naud_Pohl_Soares,NaudMagee}. A proof of Proposition \ref{Fredholm_identity} is given in \cite{JNS}. 

Proposition \ref{Fredholm_identity} has the following remarkable corollary: since $\mathcal{L}_{s,\rho}$ depends holomorphically on the variable $s$ it follows directly that $Z_{\Gamma}(s,\rho)$ extends to an entire function. This is far from obvious from the definition of the Selberg zeta function in \eqref{def_szf_twisted} as an infinite product over conjugacy classes.

\subsection{Reduced words and bounds for derivatives}\label{reduced_words_section}
We need some notations for indexing elements of the Schottky group $\Gamma$. Given a finite word
$$
\textbf{a} = a_1 \cdots a_N 
$$
with letters $a_1, \dots, a_N $ in the alphabet $ [2m] \defeq \{1,\dots,2m\}$, we set
\begin{equation*}
\gamma_{\textbf{a}} \defeq \gamma_{a_1}\circ \cdots \circ \gamma_{a_N} \in \Gamma.
\end{equation*}
Recall that $\Gamma$ is freely generated by the elements $\gamma_1,\dots, \gamma_{2m}$, which is equivalent with the generators $\gamma_1,\dots, \gamma_{2m}$ having no relations other than the trivial ones $\gamma_i^{-1} \gamma_i = \gamma_i \gamma_i^{-1} = e$. Recall also that we write $ \gamma_i^{-1} = \gamma_{i+m}$ where the indices are defined modulo $2m$. We call the word $\textbf{a}$ ``reduced'' if the corresponding element $\gamma_\textbf{a}$ is reduced when viewed as a word in the alphabet $\{ \gamma_1,\dots, \gamma_{2m} \}$. This is the case if and only if $\textbf{a} = a_1 \cdots a_N $ satisfies $a_i \neq a_{i+1} + m \;(\mathrm{mod}\; 2m)$ for all $i=1,\dots, N-1$. 

We denote by $W_N$ the set of reduced words of length $N$: 
\begin{equation}\label{reduced_words}
W_N \defeq \{ \textbf{a}=a_1\cdots a_N : \text{ $a_i \neq a_{i+1} + m \;(\mathrm{mod}\; 2m)$ for all $i=1,\dots, N-1$} \}.
\end{equation}
For $j \in [2m]$ we denote by $W_N^j$ the set of reduced words of length $N$ not ending with the letter $j$:
$$
W_N^j \defeq \{ \textbf{a}=a_1\cdots a_N \in W_N : a_N \neq j \}.
$$
With these notations in place, one verifies inductively that for all $j\in [2m]$ and all $N\in \mathbb{N}$,
\begin{equation}\label{iterates_formula}
\mathcal{L}_{s,\rho}^N F(z) = \sum_{\textbf{a}\in W_N^j} \gamma_{\textbf{a}}'(z)^s \rho(\gamma_\textbf{a})^{-1} F(\gamma_\textbf{a}(z)) \text{ if $z\in D_j$}.
\end{equation}
We record the following crucial distortion estimates, referring to Naud \cite{Naud1} for the proofs:

\begin{prop}[Distortion estimates]\label{prop:DE}
The following estimates hold true:
\begin{enumerate}[(1)]
\item (Uniform hyperbolicity) There are $c_1,c_2 >0$ and $0<\theta_1 < \theta_2 < 1$ such that for all $N$, all $j\in [2m]$ and all $\textbf{a}\in W_N^j$, we have
$$
c_1 \theta_1^{N} < \sup_{z\in D_j}  \vert \gamma'_{\textbf{a}}(z) \vert < c_2 \theta_2^{N}.
$$
\item (Bounded distortion 1) There exists $c_3 > 0$ such that for all $N$, all $j\in [2m]$, all $\textbf{a}\in W_N^j$ and all $z\in D_j$, we have
$$
\sup_{z\in D_j} \left\vert \frac{\gamma_{\textbf{a}}''(z)}{\gamma_{\textbf{a}}'(z)} \right\vert \leqslant c_3.
$$
\item (Bounded distortion 2)
There exists $c_4 > 0$ such that for all words $\textbf{a}\in W_N^j$ and all pair of points $z_1,z_2 \in D_j$, we have 
$$
\left\vert \frac{\gamma_\textbf{a}'(z_1)}{\gamma_\textbf{a}'(z_2)}  \right\vert \leqslant c_4.
$$
\end{enumerate}
All the constants depend solely upon $\Gamma$.
\end{prop}

\subsection{Topological pressure and Ruelle-Perron-Frobenius theorem}\label{sec:toppressure}
Notations being as in §\ref{Schottky_groups}, consider the map
\begin{equation*}
T \colon I \defeq \bigsqcup_{j=1}^{2m} I_j \rightarrow I, \quad T(x)=\gamma_j(x)\ \mathrm{if}\ x\in I_j.
\end{equation*}
This map, sometimes called the ``Bowen--Series map'', encodes the dynamics of the group $\Gamma$. The origins of this type of coding go back to the work of Bowen--Series \cite{BowenSeries1}. The limit set $\Lambda$ of $\Gamma$ can be re-interpreted as the non-wandering set of $T$:
\begin{equation}\label{limit_set_2}
\Lambda =\bigcap_{N=1}^{\infty} T^{-N}(I).
\end{equation}
For a continuous function $\varphi\colon I \to \R$, we define the ``topological pressure'' in terms of weighted sums over periodic orbits through the formula
\begin{equation}\label{pressure0}
P(\varphi) \defeq \lim_{N \rightarrow \infty} \frac{1}{N} \log \left( \sum_{T^N x=x} e^{\varphi^{(n)}(x)}\right),
\end{equation}
where 
\begin{equation}
\varphi^{(n)}(x) \defeq \varphi(x)+\varphi(Tx)+\ldots +\varphi(T^{n-1}x).
\end{equation}
By the variational formula it follows that
\begin{equation}
P(\varphi)=\sup_{\mu} \left ( h_\mu(T)- \int_{\Lambda} \varphi d\mu \right),
\end{equation}
where $\mu$ ranges over the set of $T$-invariant probability measures and $h_\mu(T)$ is the measure theoretic entropy. We refer the reader to \cite{ParryPollicott2} for more background on thermodynamic formalism, including facts and properties of the topological pressure. Important for us is Bowen's celebrated result \cite{Bowen1}, which says that the map
\begin{equation}\label{pressure_definition}
P \colon \R \to \R, \quad\sigma \mapsto P(\sigma) \defeq P(-\sigma \log  \vert T' \vert )
\end{equation}
is convex, strictly decreasing and vanishes precisely at $\sigma=\delta$, the Hausdorff dimension of the limit set. The relevance of the topological pressure stems from the Ruelle-Perron-Frobenius theorem:

\begin{prop}[Ruelle-Perron-Frobenius]\label{Ruelle_Perron_Frobenius} Set $\mathcal{L}_\sigma=\mathcal{L}_{\sigma,\textbf{1}_\Gamma}$ where $\sigma\in \R$ is real and $\textbf{1}_\Gamma$ is the one-dimensional trivial representation. Then the following statements hold true:
\begin{enumerate}[(1)]
\item The spectral radius of $\mathcal{L}_\sigma$ on $C^1(I)$ is equal to $e^{P(\sigma)}$ which is a simple eigenvalue associated to a strictly positive
eigenfunction $\varphi_\sigma>0$ in $C^1(I)$.

\item The operator $\mathcal{L}_\sigma$ on $C^1(I)$ is quasi-compact with essential spectral radius smaller than $\kappa(\sigma)e^{P(\sigma)}$ for some 
$\kappa(\sigma)<1$.

\item There are no other eigenvalues on $\vert z\vert=e^{P(\sigma)}$.

\item Moreover, the spectral projector $\mathbb{P}_\sigma$ on $\{e^{P(\sigma)}\}$ is given by
$$\mathbb{P}_\sigma(f)=\varphi_\sigma \int_{\Lambda} f d\mu_\sigma,$$
where $\mu_\sigma$ is the unique $T$-invariant probability measure on $\Lambda$ that satisfies
$$\mathcal{L}_\sigma^*(\mu_\sigma)=e^{P(\sigma)}\mu_\sigma.$$
\end{enumerate}
\end{prop}

For a proof, we refer to \cite[Theorem~2.2]{ParryPollicott2}. An important consequence of the Ruelle--Perron--Frobenius theorem is the following estimate, which is due to Naud \cite{Naud1}:

\begin{lemma}[Pressure estimate]\label{pressure_lemma}
For every $M>0$ there exists a constant $C>0$ such that for all $\sigma\in [0,M]$ and all $N\geqslant 1$,
$$
\sum_{j=1}^{2m} \sum_{\textbf{a} \in W_N^j} \Vert \gamma_\textbf{a}'\Vert_{\infty,D_j}^\sigma \leqslant C e^{N P(\sigma)},
$$
where
$$
\Vert g\Vert_{\infty,D_j} \defeq \sup_{z\in D_j} \vert g(z)\vert,
$$
\end{lemma}

\section{Proof of Theorem \ref{thm:dist_zero_free}}\label{sec:dist_zero_free}
In this section we prove the following result concerning the location of the zeros of twisted Selberg zeta functions, from which Theorem \ref{thm:dist_zero_free} follows directly:

\begin{theorem}[Universal zero-free region]\label{thm:dist_zero_free:reformul}
Fix a non-elementary Schottky group $\Gamma$ as in §\ref{Schottky_groups}. Then, there exists a constant $\eta_\infty > 0$ and a function $\eta\colon \mathbb{R} \to \mathbb{R} $ with
\begin{itemize}
\item $\eta(0) = 0$,
\item $\eta(t) > 0$ for $t\neq 0$, and
\item $\lim_{\vert t\vert \to \infty} \eta(t) = \eta_\infty,$
\end{itemize}
with the following property: for any finite-dimensional unitary representation $\rho \colon \Gamma \to \mathrm{U}(V)$, all the zeros of $Z_\Gamma(s,\rho)$ are contained inside
\begin{equation}\label{region_W_reformul}
\{ s = \sigma + it : \sigma \leqslant \delta -\eta(t) \}.
\end{equation}
Furthermore, we may take
$$
\eta(t) = \eta_\infty \frac{\xi(t) }{1+\xi(t)}
$$
with $\xi(t) = \frac{\vert t\vert^2}{\log \left(  C + \frac{1}{\vert t\vert} \right)}$ for some constant $C>1$ depending only on $\Gamma$.
\end{theorem}

As a consequence we obtain

\begin{cor}\label{cor:dist_zero_free}
There are constants $\eta_0 = \eta_0(\Gamma) >0 $ and $c_0 = c_0(\Gamma)>0$ with the following property: if $Z_{\Gamma}(s,\rho)$ has a zero $s=\sigma + it$ with $\sigma\geqslant \delta-\eta$, then either $\eta > \eta_0$ or $\vert t\vert \leqslant c_0 (\eta \log \eta^{-1})^{1/2}$.
\end{cor}

\begin{proof}
Let $\eta(t) = \eta_\infty \frac{\xi(t) }{1+\xi(t)}$ be as in Theorem \ref{thm:dist_zero_free:reformul}. Put $\eta_0 = \eta_\infty/2$. Suppose that $Z_{\Gamma}(s,\rho)$ has a zero at $s=\sigma + it$ with $\sigma \geqslant \delta - \eta$ and $0<\eta<\eta_0.$ Then Theorem \ref{thm:dist_zero_free:reformul} forces $\sigma\geqslant \delta-\eta(t),$ which can be rearranged to give
$$
\frac{\xi(t)}{1+\xi(t)}\leqslant \frac{\eta}{\eta_\infty}.
$$
This yields
$$
\xi(t)\leqslant \frac{\eta}{\eta_\infty - \eta} \leqslant \frac{\eta}{\eta_\infty - \eta_0} = \frac{2\eta}{\eta_\infty}.
$$
Recalling that $\xi(t) = \vert t\vert^2  \log \left(  C + \frac{1}{\vert t\vert} \right)$, the previous inequality implies that
$$
\vert t\vert \ll_\Gamma (\eta \log \eta^{-1})^{1/2},
$$
as claimed.
\end{proof}

Using the Venkov--Zograf formula \eqref{VZ_ind_formula}, it is clear how to deduce Theorem \ref{thm:dist_zero_free} from Theorem \ref{thm:dist_zero_free:reformul} above. The proof of the latter occupies the remainder of this section.

Recall that $X=\Gamma\backslash \mathbb{H}^2$ is a non-elementary convex co-compact hyperbolic quotient and that we fix a Schottky representation for $\Gamma$ as in §\ref{Schottky_groups} with generators $\gamma_1, \dots, \gamma_m$. For the rest of this section we fix a finite-dimensional, unitary representation $(\rho,V)$ of $\Gamma$ and we let $\mathcal{L}_{s,\rho}$ be the associated transfer operator as defined in §\ref{sec:transfer_operator} above.

\subsection{Normalized Transfer Operators}\label{sec:norm_to}
For every $\sigma\in \R$ let $P(\sigma)$ be the topological pressure given by \eqref{pressure_definition}. Recall that by the Ruelle--Perron--Frobenius Theorem (Proposition \ref{Ruelle_Perron_Frobenius}), the operator $\mathcal{L}_\sigma$ has maximal eigenvalue $e^{P(\sigma)}$ with strictly positive eigenfunction $\varphi_\sigma \in C^1(I)$. This eigenfunction is unique up to scaling, so for every $\sigma\in \R$ we fix once and for all $\varphi_\sigma>0$ so that $\Vert \varphi_\sigma\Vert_\infty = 1.$ We define the ``normalized'' transfer operator for any $s=\sigma+it\in \CC$ as follows:
\begin{equation}\label{normalized_defi}
\widetilde{\mathcal{L}}_{s,\rho} \defeq e^{-P(\sigma)} \varphi_\sigma^{-1} \mathcal{L}_{s,\rho} \varphi_\sigma.
\end{equation}
Since $\varphi_\sigma$ is strictly positive on $I$, this yields a well-defined operator
\begin{equation*}
\widetilde{\mathcal{L}}_{s,\rho}\colon C^1(I,V)\to C^1(I,V),
\end{equation*}
where $C^1(I,V)$ is the set of continuously differentiable functions $f\colon I \to V$. Note that for any $N\in \mathbb{N}$ the $N$-th iterate of the normalized transfer operator can be written as follows:
\begin{equation}\label{normalizedfull}
\widetilde{\mathcal{L}}_{s,\rho}^N = e^{-N P(\sigma)} \varphi_\sigma^{-1} \mathcal{L}_{s,\rho}^N \varphi_\sigma.
\end{equation}
Thus, using \eqref{iterates_formula}, we can write for all $f\in C^1(I,V)$, $j\in [2m]$ and $x\in I_j$,
\begin{equation}
\widetilde{\mathcal{L}}_{s,\rho}^N f(x) = \sum_{\textbf{a}\in W_N^j} w_{\textbf{a},\sigma}(x) \gamma_{\textbf{a}}'(x)^{it} \rho(\gamma_\textbf{a})^{-1} f(\gamma_\textbf{a}(x)),
\end{equation}
where
\begin{equation}\label{weights}
w_{\textbf{a},\sigma}(x) \defeq \frac{\varphi_\sigma(\gamma_{\textbf{a}}(x)) \gamma_{\textbf{a}}'(x)^\sigma}{\varphi_\sigma(x) e^{NP(\sigma)}} > 0.
\end{equation}
We call the operator $\widetilde{\mathcal{L}}_{s,\rho}$ ``normalized''  because
\begin{equation*}
\widetilde{\mathcal{L}}_{\sigma,\id}\textbf{1} = \textbf{1},
\end{equation*}
where $\textbf{1}$ is the constant one function on $I.$ This is equivalent with saying that for all $x\in I$ and $j\in [2m]$ with $x\in I_j$,
\begin{equation}\label{sum_to_one}
\sum_{\textbf{a}\in W_N^j} w_{\textbf{a},\sigma}(x) = 1.
\end{equation}
Now, for each $\textbf{a}\in W_N^j$, we define the vector
\begin{equation}\label{v_a}
v_\textbf{a} \defeq \gamma_{\textbf{a}}'(x)^{it} \rho(\gamma_\textbf{a})^{-1} f(\gamma_\textbf{a}(x))\in V,
\end{equation}
so that
\begin{equation}\label{convexCombiva}
\widetilde{\mathcal{L}}_{s,\rho}^N f(x) = \sum_{\textbf{a}\in W_N^j} w_{\textbf{a},\sigma}(x) v_\textbf{a}.
\end{equation}
It follows that $\widetilde{\mathcal{L}}_{s,\rho}^N f(x)$, when regarded as a vector in $V$, lies in the convex hull spanned by the vectors $v_\textbf{a}\in V$. The following result is a consequence of the fact that the unit sphere in $V$ is strictly convex.

\begin{prop}[Strict convexity]\label{prop:strict_convexity}
Let $s=\sigma+it$ with $\sigma\in [0,\delta]$, let $(\rho,V)$ be a finite-dimensional, unitary representation of $\Gamma$, and suppose $F\in H^2(D, V)$ is a 1-eigenfunction of $\mathcal{L}_{s,\rho}$. Set $f = \varphi_\sigma^{-1} F$. Choose $j\in [2m]$ and $x\in I_j$ such that
$$
\Vert f(x)\Vert_V = \Vert f\Vert_\infty.
$$
Then, notations being as above, we have
\begin{equation}\label{CSapplywei}
\sum_{\textbf{a}\in W_N^j} w_{\textbf{a},\sigma}(x) \Vert f(x) - v_\textbf{a}\Vert_V \leqslant \sqrt{C N(\delta-\sigma)} \Vert f\Vert_\infty.
\end{equation}
Moreover, for all $N\in \mathbb{N}$ and for all words $\textbf{a}\in W_N^j$,
\begin{equation}
\Vert f(x) - v_\textbf{a}\Vert_V \leqslant \sqrt{(\delta-\sigma)} e^{C'N} \Vert f\Vert_\infty.
\end{equation}
The constants $C>0$ and $C'>0$ depend only on $\Gamma.$
\end{prop}

\begin{proof}
Let $F\in H^2(D,V)$, $j\in [2m]$ and $x\in I_j$ be as in the statement. Then for all $N\in \mathbb{N}$,
\begin{equation}\label{prenormalized}
F = \mathcal{L}_{s,\rho}^N F.
\end{equation}
Let $\widetilde{\mathcal{L}}_{s,\rho}$ be the normalized operator in \eqref{normalized_defi} and let $f = \varphi_\sigma^{-1} F$. Then, equation \eqref{prenormalized} becomes
$$
e^{-NP(\sigma)} f(x) = \widetilde{\mathcal{L}}_{s,\rho}^N f(x).
$$
In particular, this gives
$$
e^{-NP(\sigma)} \Vert f(x)\Vert_V^2 = \mathrm{Re}\left\langle \widetilde{\mathcal{L}}_{s,\rho}^N f(x), f(x) \right\rangle_V
$$
Using \eqref{convexCombiva} we can expand the right-hand side to obtain
\begin{equation}\label{tobeiserted}
e^{-NP(\sigma)} \Vert f(x)\Vert_V^2 = \sum_{\textbf{a}\in W_N^j} w_{\textbf{a},\sigma}(x) \mathrm{Re} \left\langle v_\textbf{a}, f(x) \right\rangle_V.
\end{equation}
Using the identity
$$
\mathrm{Re} \left\langle v_\textbf{a}, f(x) \right\rangle_V = \frac{1}{2}\left( \Vert v_\textbf{a}\Vert_V^2 + \Vert f(x)\Vert_V^2 - \Vert f(x)- v_\textbf{a}\Vert_V^2 \right)
$$
together with the bound $\Vert  v_\textbf{a}\Vert_V \leqslant \Vert f\Vert_\infty$ for all $\textbf{a}\in W_N^j$ yields
$$
\mathrm{Re} \left\langle v_\textbf{a}, f(x) \right\rangle_V \leqslant \Vert f\Vert_{\infty}^2 - \frac{1}{2}\Vert f(x)- v_\textbf{a}\Vert_V^2.
$$
Inserting this into \eqref{tobeiserted} gives
\begin{align*}
e^{-NP(\sigma)} \Vert f(x)\Vert_V^2 &\leqslant \sum_{\textbf{a}\in W_N^j} w_{\textbf{a},\sigma}(x) \left(  \Vert f\Vert_{\infty}^2 - \frac{1}{2}\Vert f(x)- v_\textbf{a}\Vert_V^2 \right) \\ 
&\leqslant\Vert f\Vert_{\infty}^2 - \frac{1}{2} \sum_{\textbf{a}\in W_N^j} w_{\textbf{a},\sigma}(x) \Vert f(x)- v_\textbf{a}\Vert_V^2.
\end{align*}
Recall that $\Vert f(x)\Vert_V = \Vert f \Vert_\infty$ by assumption. We can thus rearrange the previous inequality to give
\begin{equation}\label{afterbeinginserted}
\sum_{\textbf{a}\in W_N^j} w_{\textbf{a},\sigma}(x) \Vert f(x)- v_\textbf{a}\Vert_V^2 \leqslant 2\left( 1- e^{-NP(\sigma)}\right) \Vert f\Vert_\infty^2.
\end{equation}
Continuity of the topological pressure $\sigma \mapsto P(\sigma)$ together with the fact that $P(\delta)=0$ yields for all $\sigma\leqslant \delta$,
\begin{equation}\label{further_dos}
1- e^{- N P(\sigma)} = e^{- N P(\delta)}  - e^{- N P(\sigma)} \leqslant C N (\delta-\sigma).
\end{equation}
for some $C=C(\Gamma)>0.$ Inserting this into \eqref{afterbeinginserted} we obtain
$$
\sum_{\textbf{a}\in W_N^j} w_{\textbf{a},\sigma}(x) \Vert f(x)- v_\textbf{a}\Vert_V^2 \leqslant 2C N (\delta-\sigma) \Vert f\Vert_\infty^2.
$$
Applying the Cauchy--Schwarz inequality yields
\begin{align*}
\sum_{\textbf{a}\in W_N^j} w_{\textbf{a},\sigma}(x) \Vert f(x)- v_\textbf{a}\Vert_V &\leqslant \left( \sum_{\textbf{a}\in W_N^j} w_{\textbf{a},\sigma}(x) \Vert f(x)- v_\textbf{a}\Vert_V^2 \right)^{1/2}\\
&\leqslant \sqrt{2C N (\delta-\sigma)} \Vert f\Vert_\infty,
\end{align*}
which proves \eqref{CSapplywei}. In particular, this inequality clearly implies that for all $\textbf{a}\in W_N^j$, 
$$
\Vert f(x) - v_\textbf{a}\Vert_V \leqslant \frac{\sqrt{C N (\delta-\sigma)}}{w_{\textbf{a}}(x)}  \Vert f\Vert_\infty.
$$
It remains to bound the weights from below. From their definition in \eqref{weights} and from the uniform hyperbolicity property in Proposition \ref{prop:DE} it follows that
\begin{equation}\label{further_uno}
w_{\textbf{a},\sigma}(x) \geqslant e^{-C'N}
\end{equation}
for some $C'=C'(\Gamma)>0$. Inserting this into \eqref{further_uno} yields
$$
\Vert f(x) - v_\textbf{a}\Vert_V \leqslant \sqrt{(\delta-\sigma)} e^{C'' N} \Vert f\Vert_\infty,
$$
for some suitable constant $C'' = C''(\Gamma)>0$, completing the proof.
\end{proof}

\subsection{Tensor power trick}\label{sec:tensorProductTrick}
Given a unitary representation $\rho \colon \Gamma \to \mathrm{U}(V)$ and $k\in \mathbb{N}$, the $k$-th tensor power
$$
\rho^{\otimes k}= \underbrace{\rho\otimes \cdots \otimes \rho}_{\text{$k$ times}}
$$
yields a new representation of $\Gamma$ with representation space
$$
V^{\otimes k} = \underbrace{V\otimes \cdots \otimes V}_{\text{$k$ times}},
$$
by defining
$$\rho^{\otimes k}(\gamma)(v_1 \otimes \cdots\otimes v_k) \defeq (\rho(\gamma)v_1 )\otimes \cdots\otimes (\rho(\gamma)v_k), \quad \gamma\in \Gamma
$$
for all $v_1, \dots,  v_k \in V$, and extending by linearity. We turn $V^{\otimes k}$ into an inner product space by defining
$$
\langle v_1 \otimes \cdots\otimes v_k , w_1 \otimes \cdots\otimes w_k\rangle_{V^{\otimes k}} \defeq \prod_{j=1}^{k}\langle v_j,w_j\rangle_V.
$$
for all $v_1, \dots,  v_k , w_1, \dots, w_k \in V$, and again extending by linearity. The representation $\rho^{\otimes k}$ is unitary with respect to this inner product. Furthermore, the norm $\Vert \cdot\Vert_{V^{\otimes k}}$ induced by $\langle \cdot, \cdot\rangle_{V^{\otimes k}}$ satisfies
\begin{equation}\label{prodtensorrule}
\Vert v_1 \otimes \cdots\otimes v_k\Vert_{V^{\otimes k}} = \prod_{j=1}^{k}\Vert v_j\Vert_V.
\end{equation}
We will drop the subscripts from the inner product and from the norm whenever it is clear which space is being referred to.

\begin{lemma}[Tensor power trick]\label{tensor_product_trick}
Let $s=\sigma+it$ with $\sigma\in [0,\delta]$ and suppose that $F\in H^2(D, V)$ is a 1-eigenfunction of $\mathcal{L}_{s,\rho}$. Set $f = \varphi_\sigma^{-1} F$. Choose $j\in [2m]$ and $x\in I_j$ such that
$$
\Vert f(x)\Vert_V = \Vert f\Vert_\infty.
$$
Fix $k\in \mathbb{N}$, let $(\nu,W) = (\rho^{\otimes k},V^{\otimes k})$ be the $k$-th tensor power representation and put $h=f^{\otimes k}$. Then we have
\begin{equation}
\left\Vert h(x) - \widetilde{\mathcal{L}}_{\nu, \sigma+itk}^N h(x)\right\Vert_W \leqslant k \sqrt{C N(\delta-\sigma)} \Vert f\Vert_{\infty}^{k},
\end{equation}
where the constant $C>0$ depends only on $\Gamma$.
\end{lemma}

We will use the following elementary estimate.

\begin{lemma}\label{elementary_tensor_bound} For all $v,w \in V$ we have
$$
\Vert v^{\otimes k} - w^{\otimes k} \Vert \leqslant k (\max\{ \Vert v\Vert, \Vert w\Vert \})^{k-1}\Vert v-w \Vert.
$$
\end{lemma}

\begin{proof}
Without loss of generality we may assume that $\Vert v\Vert \leqslant 1$ and $\Vert w\Vert \leqslant 1$. For $k=1$ the inequality is trivial. For $k\geqslant 2$ we can write
$$
v^{\otimes k} - w^{\otimes k} = v^{\otimes (k-1)} \otimes (v - w) + \left( v^{\otimes (k-1)} - w^{\otimes (k-1)} \right) \otimes w.
$$
Using the triangle inequality and \eqref{prodtensorrule} gives
\begin{align*}
\Vert v^{\otimes k} - w^{\otimes k} \Vert &\leqslant \Vert v\Vert^{k-1} \Vert v - w\Vert +  \Vert v^{\otimes (k-1)} - w^{\otimes (k-1)}\Vert \Vert w\Vert\\
&\leqslant \Vert v - w\Vert + \Vert v^{\otimes (k-1)} - w^{\otimes (k-1)}\Vert.
\end{align*}
The result now follows by induction.
\end{proof}

\begin{proof}[Proof of Lemma \ref{tensor_product_trick}]
We claim that
\begin{equation}\label{esttensor}
\left\Vert h(x) - \gamma_{\textbf{a}}'(x)^{itk} \nu(\gamma_\textbf{a})^{-1} h(\gamma_\textbf{a}(x)) \right\Vert \leqslant k \Vert f\Vert_{\infty}^{k-1}\Vert f(x)-v_\textbf{a} \Vert
\end{equation}
for all $\textbf{a}\in W_N^j$. This can be derived by applying Lemma \ref{elementary_tensor_bound} to 
$$
\text{$w = f(x)$ and $v = \gamma_{\textbf{a}}'(x)^{it} \rho(\gamma_\textbf{a})^{-1} f(\gamma_\textbf{a}(x)) $,}
$$
and observing that
$$
\text{$w^{\otimes k} = h(x)$ and $v^{\otimes k} = \gamma_{\textbf{a}}'(x)^{itk} \nu(\gamma_\textbf{a})^{-1} h(\gamma_\textbf{a}(x))$}.
$$
Now note that we can write
$$
h(x) - \widetilde{\mathcal{L}}_{\nu, \sigma+itk}^N h (x) = \sum_{\textbf{a}\in W_N^j} w_{\textbf{a},\sigma}(x)\left( h(x) - \gamma_{\textbf{a}}'(x)^{itk} \nu(\gamma_\textbf{a})^{-1} h(\gamma_\textbf{a}(x)) \right).
$$
Thus, from \eqref{esttensor} it follows that
\begin{align*}
\left\Vert h(x) - \widetilde{\mathcal{L}}_{\nu, \sigma+itk}^N h (x)\right\Vert &\leqslant  \sum_{\textbf{a}\in W_N^j} w_{\textbf{a},\sigma}(x)\left\Vert h(x) - \gamma_{\textbf{a}}'(x)^{itk} \nu(\gamma_\textbf{a})^{-1} h(\gamma_\textbf{a}(x)) \right\Vert\\
&\leqslant  k \Vert f\Vert_{\infty}^{k-1} \sum_{\textbf{a}\in W_N^j} w_{\textbf{a},\sigma}(x) \Vert f(x)-v_\textbf{a} \Vert.
\end{align*}
By Proposition \ref{prop:strict_convexity}, the sum on the last line is bounded by $\sqrt{C N(\delta-\sigma)} \Vert f\Vert_{\infty}$, so we obtain
$$
\left\Vert h(x) - \widetilde{\mathcal{L}}_{\nu, \sigma+itk}^N h(x)\right\Vert \leqslant k \sqrt{C N(\delta-\sigma)} \Vert f\Vert_{\infty}^{k}, 
$$
as claimed.
\end{proof}

\subsection{High--Frequency Norm Estimate}\label{sec:hf_norm_estimate}
We wish to estimate the $C^1$-norm of $\widetilde{\mathcal{L}}_{\delta+it,\nu}^N$ when $\vert t\vert$ is large. Throughout this section $(\nu,W)$ is a finite-dimensional, unitary representation of $\Gamma.$ Let $\langle\cdot,\cdot\rangle_W $ be the inner product of $W$ with respect to which $\nu$ is unitary, and let $\Vert v\Vert_W = \sqrt{\langle v,v\rangle_W}$ be the induced norm. Let $C^1(I,W)$ be the set of all $f\colon I\to W$ that are  continuously differentiable. We endow this space with the norm 
$$
\Vert f\Vert_{C^1(I)} = \Vert f\Vert_{\infty}+ \Vert f'\Vert_{\infty},
$$
where the maximum norm is given as usual by $\Vert f\Vert_\infty = \sup_{x\in I} \Vert f(x)\Vert_W. $ To define the derivative of $f\in C^1(I,W)$, fix an orthonormal basis $e_1,\dots, e_d$ of $W$, and write
$$
f = \sum_{k=1}^d f_k e_k
$$
where $f_k (\cdot) = \langle f(\cdot),e_k\rangle_W\in C^1(I).$ The derivative is then given by
$$
f' = \sum_{k=1}^d f_k' e_k.
$$
One easily verifies that $f'$ (and hence $\Vert f\Vert_{C^1}$) does not depend on the chosen basis.

The main result of this subsection is the following:

\begin{prop}[High-frequency norm estimate]\label{prop:hf_norm_estimate}
There are positive constants $ c, \tilde{\eta}, T_0$, depending only on $\Gamma$, such that for all $s = \sigma + it$ with $\sigma > \delta - \tilde{\eta}$ and $\vert t\vert \geqslant T_0$, the following inequality holds for $N(t) = \lceil c \log \vert t \vert \rceil$:
\begin{equation}
\Vert \widetilde{\mathcal{L}}_{\sigma+it,\nu}^{N(t)} h\Vert_{C^1(I)} \leqslant \vert t\vert^{-1} \Vert h\Vert_{C^1(I)}.
\end{equation}
We write $\lceil x\rceil$ for the least integer greater than or equal to $x$.
\end{prop}

We stress that the constants in Proposition \ref{prop:hf_norm_estimate} are independent of the representation $(\nu,W)$. 

The main technical ingredient in the proof of this proposition is a so-called ``uniform Dolgopyat'' estimate, by reference to Dolgopyat's work on Anosov flows \cite{Dol98}. Here, ``uniformity'' is meant with respect to the representation $(\nu,W)$. To state this estimate, we introduce the following warped norm on $C^1(I,W)$:
\begin{equation}
\Vert f\Vert_{(t)} \defeq \Vert f \Vert_\infty + \frac{1}{\vert t\vert} \Vert f'\Vert_\infty.
\end{equation}

\begin{prop}[Uniform Dolgopyat estimate, Proposition 7.3 in \cite{NaudMagee}]\label{prop:MN_input} There are positive constants $\beta,c,\eta_0,C,T_0,$ depending only on $\Gamma$, such that for all $f\in C^1(I,W)$, $\vert t\vert \geqslant T_0$, and $\sigma > \delta -\eta_0$, the following inequality holds with $N = \lceil c \log \vert t\vert\rceil$: 
\begin{equation}
\int_{\Lambda} \Vert \widetilde{\mathcal{L}}_{\sigma+it,\nu}^N f(x)\Vert_W^2 d\mu_{\delta}(x) \leqslant C \vert t\vert^{-\beta} \Vert f\Vert_{(t)}^2.
\end{equation}
Here, $\mu_\delta$ is the Ruelle--Perron--Frobenius measure from Proposition \ref{Ruelle_Perron_Frobenius}.
\end{prop}

We should note that Proposition \ref{prop:MN_input}, which we will use as a black box, was stated in \cite{NaudMagee} for $\mathcal{L}_{\sigma+it,\nu}$ rather than $\widetilde{\mathcal{L}}_{\sigma+it,\nu}$. However, one easily verifies that the estimate for $\widetilde{\mathcal{L}}_{\sigma+it,\nu}$ follows from the one for $\mathcal{L}_{\sigma+it,\nu}$, up to adjusting the constant $C.$ Proposition \ref{prop:MN_input} is a consequence of a a uniform non-integrability result and a deep estimate of Bourgain--Dyatlov \cite{Bourgain_Dyatlov} on oscillatory integrals over $\Lambda$.

While the paper of Magee--Naud \cite{NaudMagee} does not explicitly state Proposition \ref{prop:hf_norm_estimate}, it can be derived from Proposition \ref{prop:MN_input} together with some additional estimates provided in their work. We record the argument in this section for the reader's convenience. First we need some standard estimates.

\begin{lemma}\label{technical_estimates} 
For every finite-dimensional, unitary representation $\nu \colon \Gamma\to \mathrm{U}(W)$, for every function $h\in C^1(I,W)$, for all $t\in \mathbb{R}$, $\sigma \in [0,\delta]$, and $N\in \mathbb{N}$, the following estimates hold:
\begin{enumerate}[(1)]
\item \label{Part_1_technical}(Lasota-Yorke estimate)
$$
\Vert (\widetilde{\mathcal{L}}_{\sigma+it,\nu}^N h)'\Vert_{\infty} \ll (1+\vert t\vert) \Vert h\Vert_{\infty} + e^{-cN} \Vert h'\Vert_{\infty}
$$
\item \label{Part_2_technical}(Bound for maximum norm)
$$
\Vert \widetilde{\mathcal{L}}_{\sigma + it,\nu}^N h\Vert_{\infty} \ll e^{CN(\delta-\sigma)} \int_{\Lambda} \Vert h(x)\Vert_W d\mu_{\delta}(x)  + e^{-cN} \Vert h'\Vert_{\infty}.
$$
\end{enumerate}
The implied constants as well as the constants $C>0$ and $c>0$ depend only on $\Gamma$.
\end{lemma}

\begin{proof}
Part \eqref{Part_1_technical} is well-know in the thermodynamic literature, at least for the trivial twist $\nu = \textbf{1}$. For general $\nu$, we refer to \cite[Lemma 7.2]{NaudMagee}. The proof can be sketched as follows. Using \eqref{normalizedfull} we can write down the formula for $\widetilde{\mathcal{L}}_{\sigma+it,\nu}^N h$. We then differentiate this formula using the usual product and chain rules and apply the triangle inequality. The resulting terms can be estimated using the definition of the weights \eqref{weights} together with the bounded distortion and uniform contraction properties. Uniformity with respect to $(\nu,W)$ comes from the unitarity of representation $\nu$, which says that $\Vert \nu(\gamma)\Vert = 1$ for all $\gamma\in \Gamma.$

We now concentrate on Part \eqref{Part_2_technical}. The triangle inequality and the unitarity of $\nu$ yield
\begin{equation}\label{triangleIneqW}
\Vert \widetilde{\mathcal{L}}_{\sigma+it,\nu}^N h(x)\Vert \leqslant \widetilde{\mathcal{L}}_{\sigma}^N u(x),
\end{equation}
where 
$$
u(x) \defeq \Vert h(x)\Vert_W.
$$
Given any index $j\in [2m]$ and any two points $x,y\in I_j$, write
\begin{equation}\label{eqBefMeanVal}
\widetilde{\mathcal{L}}_{\sigma}^N u(x) = \sum_{\textbf{a}\in W_N^j} w_{\textbf{a},\sigma}(x) u(\gamma_\textbf{a}(y)) + \sum_{\textbf{a}\in W_N^j} w_{\textbf{a},\sigma}(x) \left( u(\gamma_\textbf{a}(x))-u(\gamma_\textbf{a}(x)) \right).
\end{equation}
From the mean-value theorem it follows that
\begin{equation}
\Vert u(\gamma_\textbf{a}(x))-u(\gamma_\textbf{a}(y))\Vert \ll \Vert \gamma_\textbf{a}'\Vert_{I_j, \infty} \Vert u'\Vert_\infty.
\end{equation}
Using the definition of the weights $w_{\textbf{a},\sigma}$ in \eqref{weights} together with the distortion estimates from Proposition \ref{prop:DE}, we obtain
\begin{equation}
w_{\textbf{a},\sigma}(x) \ll e^{CN(\delta-\sigma)} \gamma_{\textbf{a}}'(x)^{\delta} \ll e^{CN(\delta-\sigma)} \gamma_{\textbf{a}}'(y)^{\delta},
\end{equation}
where both $C>0$ and the implied constants depend only on $\Gamma.$ Inserting the previous two bounds into \eqref{eqBefMeanVal} yields
\begin{align*}
\widetilde{\mathcal{L}}_{\sigma}^N  u(x) &\ll e^{CN(\delta-\sigma)} \sum_{\textbf{a}\in W_N^j} \gamma_{\textbf{a}}'(y)^{\delta} u(\gamma_\textbf{a}(y)) + \sum_{\textbf{a}\in W_N^j} w_{\textbf{a},\sigma}(x) \Vert\gamma_\textbf{a}'\Vert_{I_j, \infty} \Vert u'\Vert_\infty \\
&= e^{CN(\delta-\sigma)} \mathcal{L}_{\delta}^N  u(y) + \sum_{\textbf{a}\in W_N^j} w_{\textbf{a},\sigma}(x) \Vert \gamma_\textbf{a}'\Vert_{I_j, \infty} \Vert u'\Vert_\infty.
\end{align*}
To estimate the sum on the right, recall that uniform hyperbolicity implies that for all $\textbf{a}\in W_N^j$ we have $\Vert\gamma_\textbf{a}'\Vert_{I_j, \infty} \ll e^{-cN}$ for some $c=c(\Gamma)>0$, whence
\begin{equation}
\widetilde{\mathcal{L}}_{\sigma}^N  u(x) \ll   e^{CN(\delta-\sigma)} \mathcal{L}_\delta^N u(y) + e^{-cN} \Vert u'\Vert_\infty.
\end{equation}
Integrating both sides of this bound over $I_j$ with respect to the $y$-variable against the measure $\mu_\delta$ yields
\begin{align*}
\mu_\delta(I_j) \widetilde{\mathcal{L}}_{\sigma}^N  u(x) &\ll  e^{CN(\delta-\sigma)} \int_{I_j} \mathcal{L}_\delta^N u(y) d\mu_\delta(y) + e^{-cN}\Vert u'\Vert_\infty\\
&= e^{CN(\delta-\sigma)} \int_{\Lambda} u(y) d\mu_\delta(y) + e^{-cN}\Vert u'\Vert_\infty,
\end{align*}
where for the last equality we used the defining property of $\mu_{\delta}$ from Proposition \ref{Ruelle_Perron_Frobenius}. By basic properties of $\mu_{\delta}$ it follows that $\mu_{\delta}(I_j) >0 $ for every $j\in [2m],$ so after dividing by $\mu_\delta(I_j)$ we obtain
\begin{equation}\label{abcde}
\widetilde{\mathcal{L}}_{\sigma}^N  u(x) \ll e^{CN(\delta-\sigma)} \int_{\Lambda} u(y) d\mu_\delta(y) + e^{-cN}\Vert u'\Vert_\infty.
\end{equation}
Recalling that $u(x)= \Vert h(x)\Vert_W$ we have
$$
2 u'(x) u(x) = \frac{d}{dx} u(x)^2 = 2 \mathrm{Re} \langle h(x),h'(x)\rangle_W.
$$
Applying the Cauchy--Schwarz inequality in $W$ gives $\Vert u'\Vert_\infty \ll \Vert h'\Vert_\infty.$ Hence, it follows from \eqref{triangleIneqW} and \eqref{abcde} that
$$
\Vert \widetilde{\mathcal{L}}_{\delta+it}^N h(x)\Vert_{W}  \leqslant \widetilde{\mathcal{L}}_{\sigma}^N  u(x) \ll e^{CN(\delta-\sigma)} \int_{\Lambda} \Vert h\Vert_W(y) d\mu_\delta(y) + e^{-cN}\Vert h'\Vert_\infty.
$$
As $x\in I$ was chosen arbitrarily, we can substitute the leftmost side with $\Vert \widetilde{\mathcal{L}}_{\delta+it}^N h\Vert_{\infty} = \sup_{x\in I} \Vert \widetilde{\mathcal{L}}_{\delta+it}^N h(x)\Vert_{W}$, thereby completing the proof.
\end{proof}

\begin{proof}[Proof of Proposition \ref{prop:hf_norm_estimate}]
For future reference, note that we have the following a priori bounds:
\begin{equation}\label{apririest1}
\Vert \widetilde{\mathcal{L}}_{\sigma+it,\nu}^{N} h \Vert_\infty \ll \Vert h\Vert_\infty
\end{equation}
and
\begin{equation}\label{apririest2}
\Vert (\widetilde{\mathcal{L}}_{\sigma+it,\nu}^{N} h)' \Vert_\infty \ll \vert t\vert \Vert h\Vert_{C^1(I)}.
\end{equation}
Now let $N \approx \log \vert t\vert$ be as in Proposition \ref{prop:MN_input} and let $M$ be some positive integer to be chosen further below. Using the maximum norm bound in Lemma \ref{technical_estimates} together with \eqref{apririest2} yields
\begin{align*}
\Vert \widetilde{\mathcal{L}}_{\sigma+it,\nu}^{M+N} h \Vert_\infty &= \Vert \widetilde{\mathcal{L}}_{\sigma+it,\nu}^{M} (\widetilde{\mathcal{L}}_{\sigma+it,\nu}^{N} h) \Vert_\infty\\
&\ll e^{CM(\delta-\sigma)} \int_{\Lambda} \Vert (\widetilde{\mathcal{L}}_{\sigma+it,\nu}^{N}h)(x)\Vert d\mu_\delta(x) + e^{-cM} \Vert (\widetilde{\mathcal{L}}_{\sigma+it,\nu}^{N} h)' \Vert_\infty\\
&\ll e^{CM(\delta-\sigma)} \int_{\Lambda} \Vert \widetilde{\mathcal{L}}_{\sigma+it,\nu}^{N}h(x)\Vert d\mu_\delta(x) + e^{-cM} \vert t\vert \Vert h\Vert_{C^1(I)}.
\end{align*}
Using the Cauchy--Schwarz inequality gives
\begin{equation}
\Vert \widetilde{\mathcal{L}}_{\sigma+it,\nu}^{M+N} h \Vert_\infty \ll e^{CM(\delta-\sigma)} \left( \int_{\Lambda} \Vert \widetilde{\mathcal{L}}_{\sigma+it,\nu}^{N}h(x)\Vert^2 d\mu_\delta(x) \right)^{1/2} + e^{-cM} \vert t\vert \Vert h\Vert_{C^1(I)}.
\end{equation}
By Proposition \ref{prop:MN_input} it follows that for all $\sigma < \delta$ sufficiently close to $\delta$,
\begin{equation}
\Vert \widetilde{\mathcal{L}}_{\sigma+it,\nu}^{M+N} h \Vert_\infty \ll e^{CM(\delta-\sigma)} \vert t\vert^{-\beta/2} \Vert h\Vert_{(t)} + e^{-cM} \vert t\vert \Vert h\Vert_{C^1(I)}.
\end{equation}
Hence, for any $\kappa>0$ large enough we can find some $\eta'> 0$ with the following property: for $M = \lceil \kappa \log \vert t\vert\rceil$ and $\sigma > \delta - \eta'$,
\begin{equation}\label{firstNormEst}
\Vert \widetilde{\mathcal{L}}_{\sigma+it,\nu}^{M+N} h \Vert_\infty \ll \vert t\vert^{-\beta/3} \Vert h\Vert_{(t)}.
\end{equation}
Applying the Lasota--Yorke estimate in Lemma \ref{technical_estimates} and then \eqref{apririest2} leads to
\begin{align*}
\Vert (\widetilde{\mathcal{L}}_{\sigma+it,\nu}^{2(M+N)} h)' \Vert_\infty &= \Vert (\widetilde{\mathcal{L}}_{\sigma+it,\nu}^{M+N}(\widetilde{\mathcal{L}}_{\sigma+it,\nu}^{M+N} h))' \Vert_\infty\\
&\ll \vert t\vert \Vert \widetilde{\mathcal{L}}_{\sigma+it,\nu}^{M+N} h \Vert_\infty + e^{-c(N+M)}\Vert (\widetilde{\mathcal{L}}_{\sigma+it,\nu}^{M+N} h)' \Vert_\infty\\
&\ll \vert t\vert \Vert \widetilde{\mathcal{L}}_{\sigma+it,\nu}^{M+N} h \Vert_\infty + e^{-c(N+M)} \vert t\vert^2 \Vert  h \Vert_{(t)}.
\end{align*}
Inserting \eqref{firstNormEst} into the previous line and increasing $\kappa$ if necessary, this gives
\begin{equation}\label{secondNormEst}
\Vert (\widetilde{\mathcal{L}}_{\sigma+it,\nu}^{2(M+N)} h)' \Vert_\infty \ll \vert t\vert^{1-\beta/3} \Vert  h \Vert_{(t)}.
\end{equation}
Using \eqref{firstNormEst} and \eqref{secondNormEst}, as well as the a priori estimate \eqref{apririest1}, it follows that
\begin{align*}
\Vert \widetilde{\mathcal{L}}_{\sigma+it,\nu}^{2(M+N)} h \Vert_{(t)} &= \Vert \widetilde{\mathcal{L}}_{\sigma+it,\nu}^{2(M+N)} h \Vert_{\infty}  + \frac{1}{\vert t\vert} \Vert (\widetilde{\mathcal{L}}_{\sigma+it,\nu}^{2(M+N)} h)' \Vert_\infty\\
&\ll \Vert \widetilde{\mathcal{L}}_{\sigma+it,\nu}^{M+N} h \Vert_{\infty}  + \vert t\vert^{-\beta/3} \Vert  h \Vert_{(t)}\\
&\ll \vert t\vert^{-\beta/3} \Vert  h \Vert_{(t)}.
\end{align*}
To summarize, we have shown so far that
$$
\Vert \widetilde{\mathcal{L}}_{\sigma+it,\nu}^{2(M+N)} h \Vert_{(t)} \ll \vert t\vert^{-\beta/3} \Vert  h \Vert_{(t)}.
$$
We can apply this bound repeatedly $k$ times for any $k\in \mathbb{N}$:
$$
\Vert \widetilde{\mathcal{L}}_{\sigma+it,\nu}^{2k(M+N)} h \Vert_{(t)} \ll_k \vert t\vert^{-k\beta/3} \Vert  h \Vert_{(t)}.
$$
But now, using the trivial estimate
\begin{equation}
\Vert g\Vert_{C^1(I)} \leqslant \vert t\vert \Vert g\Vert_{(t)} \leqslant \vert t\vert \Vert g\Vert_{C^1(I)},
\end{equation}
valid for any $g\in C^1(I,W)$, we get
\begin{align*}
\Vert \widetilde{\mathcal{L}}_{\sigma+it,\nu}^{2k(M+N)} h \Vert_{C^1(I)} &\leqslant \vert t\vert \Vert \widetilde{\mathcal{L}}_{\sigma+it,\nu}^{2k(M+N)} h \Vert_{(t)}\\
&\ll_k  \vert t\vert^{-k\beta/3+1} \Vert  h \Vert_{(t)}\\
&\leqslant \vert t\vert^{-k\beta/3+1} \Vert  h \Vert_{C^1(I)}.
\end{align*}
Taking $k = \lceil 6\beta^{-1}\rceil$ gives
\begin{equation}
\Vert \widetilde{\mathcal{L}}_{\sigma+it,\nu}^{2k(M+N)} h \Vert_{C^1(I)} \leqslant \vert t\vert^{-1} \Vert  h \Vert_{C^1(I)},
\end{equation}
completing the proof of Proposition \ref{prop:hf_norm_estimate}.
\end{proof}

\subsection{Basic a priori estimate}
We continue with a basic estimate. It controls the maximum norm of $f'$ in terms of the maximum norm of $f$, where $f$ is the normalized 1-eigenfunction of the twisted transfer operator.

\begin{lemma}[A priori estimate]\label{estimate_der_eigenfct} 
Let $s=\sigma+it$ with $\sigma\in [0,\delta]$, let $(\nu,W)$ be a finite-dimensional, unitary representation of $\Gamma$, and suppose that $F\in H^2(D, W)$ is a 1-eigenfunction of $\mathcal{L}_{s,\nu}$. Set $f = \varphi_\sigma^{-1} F.$ Then there are constants $A = A(\Gamma)>0$ and $C = C(\Gamma)>0$ such that
\begin{equation*}
\Vert f'\Vert_\infty \leqslant C (1+\vert t\vert)^{A} \Vert f\Vert_{\infty}.
\end{equation*}
\end{lemma}

\begin{proof}
Consider a 1-eigenfunction $F$ of $\mathcal{L}_{s,\nu}$. For any $N\in \mathbb{N}$ we have
$$
\mathcal{L}_{s,\nu}^N F = F.
$$
Upon normalizing this can be expressed equivalently as
$$
e^{-NP(\sigma)} f = \widetilde{\mathcal{L}}_{s,\nu} f.
$$
Using the Lasota-Yorke estimate (Part \ref{Part_1_technical} of Lemma \ref{technical_estimates}) we get
$$
e^{-NP(\sigma)} \Vert f'\Vert_\infty \leqslant C' (1+\vert t\vert) \left( \Vert f\Vert_{\infty} + e^{-(c+P(\sigma))N} \Vert f'\Vert_\infty \right)
$$
for some constants $C' = C'(\Gamma)>0$ and $c=c(\Gamma)>0$, or equivalently, 
$$
\Vert f'\Vert_\infty \leqslant C' (1+\vert t\vert) \left( e^{N P(\sigma)} \Vert f\Vert_{\infty} + e^{-cN} \Vert f'\Vert_{\infty} \right).
$$
Taking $N = \lceil C'' \log (1+\vert t\vert)\rceil$ with $C''=C''(\Gamma) > 0$ so large that $$ C' (1+\vert t\vert) e^{-cN} < \frac{1}{2} $$ holds, we obtain
$$
\Vert f'\Vert_\infty \leqslant C' (1+\vert t\vert)  e^{N P(\sigma)} \Vert f\Vert_{\infty} + \frac{1}{2}  \Vert f'\Vert_{\infty}.
$$
This can be rearranged to give
$$
\Vert f'\Vert_\infty \leqslant 2 C' (1+\vert t\vert) e^{N P(\sigma)} \Vert f\Vert_{\infty} \leqslant C_0 (1+\vert t\vert)^{A}\Vert f\Vert_{\infty},
$$
for some $A=A(\Gamma)>0$, as claimed.
\end{proof}

\subsection{Finishing the proof of Theorem \ref{thm:dist_zero_free:reformul}}\label{sec:finishingDistZero}
We now finish the proof of Theorem \ref{thm:dist_zero_free:reformul}. Fix a unitary representation $\rho\colon \Gamma \to \mathrm{U}(V)$ and suppose that $Z_\Gamma(s,\rho)$ vanishes at $s=\sigma + it$. By Proposition \ref{Fredholm_identity} and standard properties of Fredholm determinants, zeros of $Z_\Gamma(s,\rho)$ correspond to non-trivial 1-eigenfunctions of the operator $\mathcal{L}_{s,\rho}$, i.e., there is some non-zero $F\in C^1(I,V)$ satisfying
\begin{equation}\label{unoProof}
\mathcal{L}_{s,\rho}^N F = F
\end{equation}
for all $N\in \mathbb{N}$. Writing $f = \varphi_\sigma^{-1} F$ and letting $\widetilde{\mathcal{L}}_{s,\rho}$ denote the normalized operator as in §\ref{sec:norm_to}, equation \eqref{unoProof} becomes
\begin{equation}\label{unoProofmod}
e^{-N P(\sigma)} f = \widetilde{\mathcal{L}}_{s,\rho}^N f.
\end{equation}
We may assume without loss of generality that $t\neq 0.$ Let $T_0$ be the constant from Proposition \ref{prop:hf_norm_estimate}. We may assume for the rest of this proof that $T_0>1.$ Let us first deal with the case $\vert t\vert \geqslant T_0>1$ Then, by Proposition \ref{prop:hf_norm_estimate}, there are $\tilde{\eta} > 0$ and $c>0$, depending only on $\Gamma$, such that for all $\sigma > \delta - \tilde{\eta}$ the following inequality holds with $N_0 = \lceil c \log \vert t\vert\rceil$:
$$
\Vert \widetilde{\mathcal{L}}_{s,\rho}^{N_0} f \Vert_{C^1(I,V)} \leqslant \vert t\vert^{-1} \Vert f\Vert_{C^1(I,V)}.
$$
Combining this with \eqref{unoProofmod} and using the fact that $f\neq 0$ it follows that $e^{- N_0 P(\sigma)} \leqslant \vert t\vert^{-1}$. Upon rearranging, this yields
\begin{equation}
1 \leqslant \vert t\vert^{-1+cP(\sigma)},
\end{equation}
which implies
\begin{equation}\label{forces}
P(\sigma) \geqslant c^{-1}.
\end{equation}
Since $\sigma \mapsto P(\sigma)$ is strictly decreasing with a zero at $\sigma = \delta$, this forces
\begin{equation}\label{part1Gap}
\sigma \leqslant \delta -\eta_1
\end{equation}
for some constant $\eta_1 = \eta_1(\Gamma) > 0$. 

Now assume alternatively that $\vert t\vert < T_0.$ In this case we apply the tensor power trick in §\ref{sec:tensorProductTrick}. Put
$$
k \defeq \left\lceil \frac{2T_0}{\vert t\vert} \right\rceil \in \mathbb{N}
$$
and define the $k$-th tensor powers
\begin{equation}\label{choice_h}
\text{$(\nu,W) \defeq (\rho^{\otimes k}, V^{\otimes k})$ and $h \defeq f^{\otimes k} \in C^1(I,W)$.}
\end{equation}
Clearly, we have $\frac{2T_0}{\vert t\vert} \leqslant k \leqslant \frac{2T_0}{\vert t\vert}+1 $, so $\vert t\vert < T_0$ implies
$$
2 < 2T_0 \leqslant k\vert t\vert \leqslant 3 T_0.
$$
Hence, by Proposition \ref{prop:hf_norm_estimate}, there are $N_0 = O(1)$ and $\tilde{\eta} > 0$ with the following property: for all $\sigma > \delta-\tilde{\eta}$ and all $g\in C^1(I,W)$,
$$
\Vert \widetilde{\mathcal{L}}_{\sigma+itk,\nu}^{N_0} g\Vert_{C^1(I)} \leqslant (k \vert t\vert)^{-1} \Vert g\Vert_{C^1(I)} \leqslant \frac{1}{2}\Vert g\Vert_{C^1}.
$$
Applying this repeatedly $l$ times for any $l\in \mathbb{N}$ it follows that
\begin{equation}\label{theBound}
\Vert \widetilde{\mathcal{L}}_{\sigma+itk,\nu}^{l N_0} g\Vert_{C^1(I)} \leqslant 2^{-l} \Vert g\Vert_{C^1(I)}.
\end{equation}
To proceed, let us estimate the $C^1$-norm of $h$. Given vector spaces $V_1$ and $V_2$, observe that by the definition of the tensor product, the following holds for all $f_1 \in C(I,V_1)$ and $f_2 \in C(I,V_2)$:
\begin{equation}\label{trule}
\Vert f_1 \otimes f_2\Vert_\infty = \Vert f_1 \Vert_\infty \Vert f_2\Vert_\infty.
\end{equation}
By repeatedly applying this rule to $h = f^{\otimes k}$ we obtain
\begin{equation}
\Vert h\Vert_{\infty} = \Vert f\Vert_{\infty}^k.
\end{equation}
Moreover, the derivative of $h$ is equal to
$$
h' = \sum_{i=1}^k h_i, \quad h_i \defeq \underbrace{f\otimes\cdots \otimes f' \otimes \cdots\otimes f}_{\text{derivative at $i$-th place}}.
$$
This follows from applying the product rule to the following identity: for any $\textbf{v} = v_1 \otimes \cdots \otimes v_k$ and $\textbf{w} = w_1 \otimes \cdots \otimes w_k$, the inner product $\langle  h(x) \textbf{v}, \textbf{w} \rangle$ can be expressed as the product of individual inner products $\langle  f(x) v_i, w_i \rangle$:
$$
\langle  h(x) \textbf{v}, \textbf{w} \rangle = \prod_{i=1}^k \langle  f(x) v_i, w_i \rangle.
$$
Using \eqref{trule} once again it follows that $\Vert h_i\Vert_\infty \leqslant \Vert f'\Vert_\infty \Vert f\Vert_\infty^{k-1}$ and so the triangle inequality yields
\begin{equation}\label{boundForderivH}
\Vert h'\Vert_\infty \leqslant \sum_{i=1}^k \Vert h_i\Vert_\infty \leqslant k \Vert f'\Vert_\infty \Vert f\Vert_\infty^{k-1}.
\end{equation}
By Lemma \ref{estimate_der_eigenfct} we have
$$
\Vert f'\Vert_\infty \ll \Vert f\Vert_\infty,
$$
which when inserted into \eqref{boundForderivH} yields
$$
\Vert h'\Vert_\infty \ll k \Vert f\Vert^k.
$$
We conclude that the $C^1$-norm of $h$ is bounded by 
\begin{equation}
\Vert h\Vert_{C^1(I)} = \Vert h\Vert_\infty + \Vert h'\Vert_\infty  \ll k \Vert f\Vert_{\infty}^k.
\end{equation}
Next, using this bound and applying \eqref{theBound} to $h$ yields
\begin{equation}
\Vert \widetilde{\mathcal{L}}_{\sigma+itk,\nu}^{l N_0} h\Vert_{C^1(I)} \ll 2^{-l} \Vert h\Vert_{C^1(I)} \ll k 2^{-l} \Vert f\Vert_{\infty}^k.
\end{equation}
We may therefore choose some $l\in \mathbb{N}$ of size at most
$$
l = O(\log k) = O(\log (T_0\vert t\vert^{-1}))
$$
such that
\begin{equation}\label{halfnorm}
\Vert \widetilde{\mathcal{L}}_{\sigma+itk,\nu}^{l N_0} h\Vert_{C^1(I)} \leqslant \frac{1}{2} \Vert f\Vert_{\infty}^k.
\end{equation}
Now choose $x\in I$ for which $\Vert f(x)\Vert_V = \Vert f\Vert_\infty$. This choice also gives
$$
\left\Vert h(x) \right\Vert_W = \left\Vert f(x) \right\Vert_V^k = \Vert f\Vert_\infty^k.
$$
The triangle inequality and \eqref{halfnorm} yield
\begin{align*}
\left\Vert h(x) - \widetilde{\mathcal{L}}_{\sigma+itk,\nu}^{N_0 l} h(x) \right\Vert_W &\geqslant \Vert h(x)\Vert_W - \Vert \widetilde{\mathcal{L}}_{\sigma+itk,\nu}^{l N_0} h(x) \Vert_W\\
&\geqslant \Vert f\Vert_\infty^k - \Vert \widetilde{\mathcal{L}}_{\sigma+itk,\nu}^{l N_0} h(x) \Vert_W\\
&\geqslant \frac{1}{2} \Vert f\Vert_{\infty}^k.
\end{align*}
On the other hand, Lemma \ref{tensor_product_trick} tells us that
$$
\left\Vert h(x) - \widetilde{\mathcal{L}}_{\sigma+itk,\nu}^{N_0 l} h(x) \right\Vert_W \leqslant k \sqrt{C N_0 l(\delta-\sigma)} \Vert f\Vert_{\infty}^k. 
$$
Combining the previous two estimates, we arrive at
\begin{equation}
\frac{1}{2} \Vert f\Vert_{\infty}^k \leqslant k \sqrt{C N_0 l(\delta-\sigma)} \Vert f\Vert_{\infty}^k.
\end{equation}
Dividing both sides by $\Vert f\Vert_{\infty}^k \neq 0$ yields
\begin{equation}
\frac{1}{2}  \leqslant k \sqrt{C N_0 l(\delta-\sigma)},
\end{equation}
or equivalently,
\begin{equation}\label{part2Gap}
\sigma \leqslant \delta - \frac{1}{4 C k^2 N_0 l}.
\end{equation}
Finally, by our choices of $k$ and $l$ we obtain
$$
\frac{1}{4 C k^2 N_0 l} = O_{\Gamma}\left( \frac{\vert t\vert^2}{\log \vert t \vert^{-1}}\right)
$$ 
as $\vert t\vert \to 0$. Combining \eqref{part1Gap} and \eqref{part2Gap}, and taking $\eta_\infty = \min\{ \eta_1,\tilde{\eta}\}$, we conclude that
$$
\sigma \leqslant \delta - \eta(t),
$$
for some function $\eta\colon \mathbb{R}\to \mathbb{R}$ with the following properties:
\begin{itemize}
\item $\eta(0)=0$ and $\eta(t) > 0$ for all $t\neq 0$,
\item $\lim_{\vert t\vert\to \infty} \eta(t) = \eta_\infty$, and
\item $\eta(t) = O_{\Gamma}\left( \frac{\vert t\vert^2}{\log \vert t \vert^{-1}}\right) $ as $\vert t\vert \to 0$.
\end{itemize}
For definiteness, we may take
$$
\eta(t) = \eta_\infty \frac{\xi(t)}{1+\xi(t)}
$$ with $\xi(t) = \frac{\vert t\vert^2}{\log \left(  C + \frac{1}{\vert t\vert} \right)}$ for some suitably large constant $C >1$ depending only upon $\Gamma.$ The proof of Theorem \ref{thm:dist_zero_free:reformul} is complete.

\section{Exploiting the expansion property}\label{sec:exploiting_expander}
Recall the notation from §\ref{sec:preliminaries}. Let $\Gamma$ is a non-elementary Schottky group with generators $S=\{ \gamma_1, \dots, \gamma_m \}$. Let us recall some of the notation. We put $I_j = D_j\cap \R$ where $D_1,\dots, D_{2m}$ are the Schottky disks in the construction of $\Gamma$ and $I = \bigsqcup_{j=1}^{2m} I_j$. The non-identity elements of $\Gamma$ are in bijection to set of reduced words $W \defeq \bigsqcup_{N=0}^\infty W_N$ in the alphabet $[2m]$, where $W_N$ is the set of \textit{reduced} words of length $N$,
$$
W_N = \{ \textbf{a}=a_1\cdots a_N : \text{ $a_i \neq a_{i+1} + m \;(\mathrm{mod}\; 2m)$ for all $i=1,\dots, N-1$} \},
$$
We denote by $W_N^j \subseteq W^N$ the subset of those words not ending with the letter $j$. Moreover, we use $\mathbb{E}_{a\in A} f(a) $ as a shorthand for the average
$$
\frac{1}{\vert A\vert} \sum_{a\in A} f(a).
$$
The goal of this section is to prove estimates for averages of the type
\begin{equation*}
 \mathbb{E}_{\textbf{a}\in W_N^j} \left( \rho(\gamma_\textbf{a})^{-1} h(\gamma_\textbf{a}(x)) \right),
\end{equation*}
where $\rho\colon \Gamma\to \mathrm{U}(V)$ is a finite-dimensional unitary representation and $h\in C^1(I,V)$. These estimates will be crucial in the proof of Theorem \ref{main_theorem}.

\subsection{Representations and sums over words}
Given a finite-dimensional unitary representation $(\rho,V)$ of $\Gamma$ and a finite set of reduced words $Z\subset W$, put
$$
Z[\rho] \defeq \sum_{\textbf{a}\in Z} \rho(\gamma_\textbf{a}).
$$
Note that $Z[\rho]$ is an element of the endomorphism ring $\mathrm{End}(V)$. Applying the triangle inequality yields
$$
\Vert Z[\rho]\Vert \leqslant \vert Z\vert,
$$ 
where we write $\Vert\cdot\Vert = \Vert \cdot\Vert_{V\to V}$ for the usual operator norm.

\begin{prop}[Norm estimate for $W_N$]\label{main_expanding_prop}
Assume that the self-adjoint operator 
$$
T(\rho) = \frac{1}{2m}\sum_{j=1}^{2m} \rho(\gamma_j) = \frac{1}{2m}\sum_{j=1}^{m} \left(  \rho(\gamma_j)+\rho(\gamma_j)^{-1}\right)\in \mathrm{End}(V)
$$
has all its eigenvalues inside $[-1+\epsilon, 1-\epsilon]$. Then the operator norm of $W_N [\rho]$ is bounded by
\begin{equation*}
\Vert W_N [\rho]\Vert \ll e^{-\frac{\epsilon}{4} N} \vert W_N\vert
\end{equation*}
with implied constant depending on $m$ and $\epsilon$ but not on $N.$
\end{prop}

\begin{proof}
Clearly $W_1[\rho] = 2m\, T(\rho). $ Thus, by assumption, the eigenvalues $\lambda_1,\dots,\lambda_d$ of $W_1[\rho]$ satisfy the bound
\begin{equation}\label{eigenvalue_bound_2}
\max_{1\leqslant k\leqslant d} \vert \lambda_k\vert \leqslant 2m(1-\epsilon),
\end{equation}
which we will use later in the proof.

Consider the following formal power series with coefficients in the endomorphism ring $ \mathrm{End}(V):$
\begin{equation}\label{formal_series}
G_\rho (z) \defeq \sum_{N=0}^\infty W_N[\rho] z^N,
\end{equation}
where $W_0[\rho] \defeq I_V$ is the identity operator. An elementary argument shows that
$$
\vert W_N\vert = 2m (2m-1)^{N-1} \asymp (2m-1)^N. 
$$
Together with the trivial bound
$$
\Vert W_N[\rho]\Vert \leqslant \vert W_N\vert,
$$
it follows that the series in \eqref{formal_series} is absolutely convergent with respect to the operator norm whenever $\vert z\vert < \frac{1}{2m-1}$.

Next we derive a closed-form expression for $G_\rho (z)$. To that effect we use the elementary but essential recursion formulas\footnote{These recursion formulas were used by Ihara \cite{Ihara} in his derivation of a closed-form expression of what came to be known as the Ihara zeta function}:
\begin{equation}\label{recursion:1}
W_1[\rho]^2 = W_2[\rho] + 2m \cdot I_V
\end{equation}
and for all $N\geqslant 2$,
\begin{equation}\label{recursion:2}
W_1[\rho] W_N[\rho] = W_{N+1}[\rho] +  (2m-1) W_{N-1}[\rho].
\end{equation}
Multiplying $ G_\rho(z) $ on the right by $ z W_1[\rho] $ and applying \eqref{recursion:1} and \eqref{recursion:2} gives us
\begin{align}
zW_1[\rho] G_\rho (z) &= W_1[\rho] z + W_1[\rho]^2 z^2 + \sum_{N=2}^\infty W_1[\rho] W_N[\rho] z^{N+1}\\
&= W_1[\rho] z + (W_2[\rho] + 2m \cdot I_V ) z^2 \\
&\qquad \qquad + \sum_{N=2}^\infty  W_{N+1}[\rho] z^{N+1} + (2m-1)\sum_{N=2}^\infty  W_{N-1}[\rho] z^{N+1} \label{recurreq}.
\end{align}
Note that 
$$
\sum_{N=2}^\infty  W_{N+1}[\rho] z^{N+1} = G_\rho (z) - W_2[\rho] z^2 - W_1[\rho] z - I_V
$$
and
$$
\sum_{N=2}^\infty  W_{N-1}[\rho] z^{N+1} = z^2 \left( G_\rho (z) - I_V \right).
$$
Inserting these two equations back into \eqref{recurreq} we arrive at
\begin{equation}\label{to_be_rearranged}
zW_1[\rho] G_\rho (z) = (z^2 - 1) I_V + G_\rho (z) \left( 1+(2m-1)z^2 \right),
\end{equation}
which, upon rearranging, gives the closed-form expression
\begin{equation*}
G_\rho (z) = (1-z^2) \left( I_V - W_1[\rho]z+(2m-1)z^2 \right)^{-1}.
\end{equation*}
Since $W_1[\rho]$ is a self-adjoint operator there exists a basis of $V$ with respect to which it acts by the diagonal matrix
$$
W_1[\rho] = \diag ( \lambda_1, \dots, \lambda_d ).
$$
With respect to this basis we have the expression
\begin{equation}\label{powerseriesclosed}
G_\rho (z) = \diag\left( \frac{1-z^2}{1- \lambda_k z+(2m-1)z^2}  \right)_{k=1}^d.
\end{equation}
For each $k\in [d]$ let $\omega^{\pm}_k$ be the pair of complex numbers satisfying 
$$
1- \lambda_k z+(2m-1)z^2 = (1-\omega_k^{+} z) (1-\omega_k^{-} z).
$$
We can then expand each diagonal entry of \eqref{powerseriesclosed} as a power series in $z$ as follows:
$$
\frac{1-z^2}{1- \lambda_k z+(2m-1)z^2} = \frac{1-z^2}{(1-\omega_k^{+} z) (1-\omega_k^{-} z)}  = \sum_{N=0}^{\infty} (\xi_{k,N} - \xi_{k,N-2}) z^N,
$$
where $\xi_{k,-1} = \xi_{k,-2} = 0$ and
\begin{equation}\label{c_k}
\xi_{k,N} = \sum_{l=0}^N \left(\omega_k^{+} \right)^l \left(\omega_k^{-} \right)^{N-l} 
\end{equation}
for all non-negative integers $N$. Comparing coefficients in \eqref{powerseriesclosed} and \eqref{formal_series} yields
$$
W_N[\rho] = \diag\left( \xi_{k,N} - \xi_{k,N-2} \right)_{k=1}^d.
$$
Since the operator norm of a diagonal matrix is equal to the largest absolute values of its entries we obtain
\begin{equation}\label{norm_W_nu}
\Vert W_N[\rho]\Vert  = \max_{1\leqslant k\leqslant d} \vert \xi_{k,N} - \xi_{k,N-2} \vert.
\end{equation}
A direct computation shows that
\begin{equation}\label{direct_computation}
\omega_k^{\pm} = \frac{\lambda_k \pm \sqrt{\lambda_k^2 - 4(2m-1)}}{2}.
\end{equation}
If $\vert \lambda_k\vert \leqslant 2 \sqrt{2m-1}$ it follows directly from this formula that
$$
\vert \omega_k^{\pm}\vert = \sqrt{2m-1}.
$$
Now suppose alternatively that $\vert \lambda_j\vert > 2 \sqrt{2m-1}$. In this case we simply recall from \eqref{eigenvalue_bound_2} that 
$$
\vert \lambda_j\vert \leqslant 2m(1-\epsilon),
$$ 
which implies that
$$
\vert \omega_k^{\pm} \vert \leqslant \frac{\vert \lambda_k\vert + \sqrt{(2m)^2 - 4(2m-1)}}{2} \leqslant (2m-1)\left( 1-\frac{m}{2m-1} \epsilon \right).
$$
In any case we have
\begin{equation}
\vert \omega_k^{\pm} \vert \leqslant (2m-1)(1-\frac{\epsilon}{3}) < (2m-1)e^{-\frac{\epsilon}{3}}.
\end{equation}
Therefore, by \eqref{c_k},
$$
\max_{1\leqslant k\leqslant d} \vert \xi_{k,N} \vert \leqslant N \left( \max_{1\leqslant k\leqslant d} \vert \omega_k^{\pm} \vert \right)^N \leqslant N e^{-\frac{\epsilon}{3}N} (2m-1)^N \ll e^{-\frac{\epsilon}{4} N} \vert W_N\vert.
$$
Returning to \eqref{norm_W_nu} we therefore end up with
$$
\Vert W_N[\rho]\Vert  \leqslant \max_{1\leqslant k\leqslant d} \vert \xi_{k,N} \vert + \max_{1\leqslant k\leqslant d} \vert \xi_{k,N-2} \vert \ll  e^{-\frac{\epsilon}{4} N} \vert W_N\vert,
$$
which completes the proof of Proposition \ref{main_expanding_prop}.
\end{proof}

For every $N\in \mathbb{N}$ and for every pair $(i,j)\in [2m]^2$ consider the set of reduced words of length $N$ starting with $i$ and ending with $j$:
\begin{equation}\label{A_ij}
A_{N}^{i,j} = \{ \textbf{a} = a_1\cdots a_N \in W_N : a_1 = i, a_N = j \}.
\end{equation}
It turns out that in our application it is more useful to have an upper bound for the norm of $A_{N}^{i,j}[\rho]$ rather than one for $W_N[\rho]$.

\begin{cor}[Norm estimate for $A_{N}^{i,j}$]\label{cor_words_decay}
Let $(\rho, V)$ be a unitary representation of $\Gamma$ and assume that $T(\rho)$ has all its eigenvalues in $[-1+\epsilon, 1-\epsilon]$. For all $N\in \mathbb{N}$ and for all $(i,j)\in [2m]^2$ the operator norm of $ A_{N}^{i,j}[\rho]$ is bounded by
\begin{equation*}
\Vert A_{N}^{i,j}[\rho] \Vert  \ll e^{-\frac{\epsilon}{6} N} \vert W_N\vert,
\end{equation*}
where the implied constant depends only on $m$ and $\epsilon$.
\end{cor}

\begin{proof} 
Fix $i\in [2m]$ and consider the set of reduced words of length $N$ starting with the letter $i,$
\begin{equation*}
B_{N}^i = \{ \textbf{a} = a_1\cdots a_N \in W_N : a_1 = i \}.
\end{equation*}
Observe that
\begin{equation*}
B_{N}^i[\rho] = \rho(\gamma_i)\left( W_{N-1}[\rho] - B_{N-1}^{i+m}[\rho] \right).
\end{equation*}
Applying the triangle inequality and using the fact that $\rho$ is unitary yields
\begin{equation}\label{triangle_ineq_ind_norm_1}
\Vert B_{N}^i[\rho]\Vert  \leqslant \Vert W_{N-1}[\rho] \Vert  + \Vert B_{N-1}^{i+m}[\rho] \Vert .
\end{equation}
Hence, setting
\begin{equation}
\mathfrak{b}_N(\rho) \defeq \max_{1\leqslant i\leqslant 2m} \Vert B_{N}^i[\rho]\Vert,
\end{equation}
gives
\begin{equation*}
\mathfrak{b}_N (\rho) \leqslant \Vert W_{N-1}[\rho] \Vert  + \mathfrak{b}_{N-1}(\rho).
\end{equation*}
Applying this inequality repeatedly yields
\begin{equation}\label{iterating_b}
\mathfrak{b}_N(\rho) \leqslant \sum_{k=0}^{N-1} \Vert W_{k}[\rho] \Vert .
\end{equation}
By Proposition \ref{main_expanding_prop},
$$
\Vert W_{k}[\rho] \Vert  \ll e^{-\frac{\epsilon}{4} k}\vert W_k\vert \ll e^{-\frac{\epsilon}{4} k}(2m-1)^k,
$$
where the implied constant does not depend on $k$. Inserting this into \eqref{iterating_b}, and using a geometric series summation together with $\vert W_N\vert \asymp (2m-1)^N$, yields
\begin{equation}\label{b_bound}
\mathfrak{b}_N (\rho) \ll \sum_{k=0}^{N-1} e^{-\frac{\epsilon}{4} k}(2m-1)^k \ll e^{-\frac{\epsilon}{5} N}(2m-1)^N \ll e^{-\frac{\epsilon}{5} N} \vert W_N\vert.
\end{equation}
Now fix $(i,j)\in [2m]^2$, let $A_{N}^{i,j}$ be as in \eqref{A_ij}, and observe that
\begin{equation}\label{second_rel}
A_{N}^{i,j}[\rho] = \left( B_{N-1}^{i}[\rho] - A_{N}^{i,j+m}[\rho]  \right) \rho(\gamma_j).
\end{equation}
Thus, setting
$$
\mathfrak{a}_{N}(\rho) \defeq \max_{1\leqslant i,j\leqslant 2m} \Vert A_{N}^{i,j}[\rho]\Vert,
$$
it follows from \eqref{second_rel} that
$$
\mathfrak{a}_N(\rho) \leqslant \mathfrak{b}_{N-1}(\rho) + \mathfrak{a}_{N-1}(\rho).
$$
Applying this bound repeatedly gives
$$
\mathfrak{a}_N(\rho) \leqslant \sum_{k=0}^{N-1} \mathfrak{b}_{k}(\rho).
$$
Using \eqref{b_bound} and arguing as above it follows that
$$
\mathfrak{a}_N(\rho) \leqslant \sum_{k=0}^{N-1} e^{-\frac{\epsilon}{5} k} \vert W_k\vert \ll e^{-\frac{\epsilon}{6} N} \vert W_N\vert,
$$
concluding the proof of Corollary \ref{cor_words_decay}.
\end{proof}

\subsection{Exponential decay for twisted averages}
We now state and prove the key result of this section.

\begin{prop}[Exponential decay for twisted averages]\label{prop:exp_dec_average}
Let $(\rho, V)$ be a unitary representation of $\Gamma$. Suppose that the operator $T(\rho)$ has all of its eigenvalues within the interval $[-1+\epsilon, 1-\epsilon]$. Then, for all $h\in C^1(I,V)$, for all $x\in I$ and for all $j\in [2m]$ such that $x\in I_j$, the following inequality holds:
\begin{equation*}
\left\Vert \mathbb{E}_{\textbf{a}\in W_N^j} \left( \rho(\gamma_\textbf{a})^{-1} h(\gamma_\textbf{a}(x)) \right)\right\Vert_V \ll e^{-c \epsilon N} \Vert h\Vert_{C^1}.
\end{equation*}
The implied constant and $c>0$ depend solely on $\Gamma.$
\end{prop}

\begin{proof}
We define for each pair of indices $(l,j)\in [2m]^2$ the following set:
$$
Z_{N}^{l,j} \defeq \{ \textbf{a} = a_1 \cdots a_N \in W_N : \text{$a_1 = l+m$ and $a_N \neq j$} \},
$$
This set comprises all words $\textbf{a}\in W_N$ such that $\gamma_\textbf{a}$ maps $D_j$ into the interior of $D_l$. We can re-express it as
$$
Z_{N}^{l,j} = \bigsqcup_{\substack{ 1\leqslant j'\leqslant 2m \\ j'\neq j }} A_{N}^{l+m,j'},
$$
where $A_{N}^{i,j}$ are the sets that we have introduced in \eqref{A_ij}. Thus, by Corollary \ref{cor_words_decay},
$$
\left\Vert Z_{N}^{l,j}[\rho]\right\Vert  \leqslant \sum_{\substack{ 1\leqslant j'\leqslant 2m \\ j'\neq j }} \left\Vert A_{N}^{l+m,j'}[\rho]\right\Vert  \ll e^{- \frac{\epsilon}{6} N} \vert W_N\vert. 
$$
One easily verifies that $Z_{N}^{l,j}$ and $W_{N}^{j}$ are of comparable size, i.e., $\vert Z_{N}^{l,j}\vert  \asymp (2m-1)^N \asymp\vert W_{N}^{j}\vert$ with implied constants independent of $N$. It follows that
\begin{equation}\label{apply_cor}
\left\Vert \mathbb{E}_{\textbf{a}\in Z_{N}^{l,j}} \left( \rho(\gamma_{\textbf{a}}) \right) \right\Vert  = \frac{1}{\vert Z_{N}^{l,j} \vert} \left\Vert Z_{N}^{l,j}[\rho]\right\Vert  \ll \frac{1}{\vert W_{N}^{j} \vert} \left\Vert Z_{N}^{l,j}[\rho]\right\Vert  \ll  e^{- \frac{\epsilon}{6} N}.
\end{equation}
We will use this estimate further below.

To proceed we use a decoupling argument similar to the one contained in \cite[Section~8]{BGS}. Write $N = N_1 + N_2,$ where $N_1$ and $N_2$ will be chosen further below, and observe that for every $j\in [2m]$ and $\textbf{a}\in W_N^j$ there are unique $l\in [2m]$ and $\textbf{a}_1\in W_{N_1}^l$ and $\textbf{a}_2\in Z_{N_2}^{l,j}$ such that $\textbf{a} = \textbf{a}_1 \textbf{a}_2$, where $\textbf{a}_1 \textbf{a}_2$ stands for the concatenation of $\textbf{a}_2$ and $\textbf{a}_2$. Therefore, for all $x\in I$,
$$
\sum_{\textbf{a}\in W_N^j} \rho(\gamma_\textbf{a})^{-1} h(\gamma_\textbf{a}(x))  = \sum_{\textbf{a}_1\in W_{N_1}^l} \sum_{\textbf{a}_2\in Z_{N_2}^{l,j}} \rho(\gamma_{\textbf{a}_1 \textbf{a}_2})^{-1} h(\gamma_{\textbf{a}_1 \textbf{a}_2} (x)) .
$$
Note that
\begin{equation*}
\vert W_N^j\vert \asymp \vert W_{N_1}^l\vert \vert Z_{N_2}^{l,j}\vert
\end{equation*}
with implied constants depending only on $m,$ so we obtain
\begin{equation}\label{sum_over_l}
\left\Vert \mathbb{E}_{\textbf{a}\in W_N^j} \left( \rho(\gamma_\textbf{a})^{-1} h(\gamma_\textbf{a}(x)) \right) \right\Vert_V \ll  \sum_{l=1}^{2m}  \underbrace{\mathbb{E}_{\textbf{a}_1\in W_{N_1}^l} \left\Vert \mathbb{E}_{\textbf{a}_2\in Z_{N_2}^{l,j}}\left(   \rho(\gamma_{\textbf{a}_1 \textbf{a}_2})^{-1} h(\gamma_{\textbf{a}_1 \textbf{a}_2} (x)) \right) \right\Vert_V}_{\defeq \Sigma_l}.
\end{equation}
It suffices to establish the bound
$$
\Sigma_l \ll e^{-c \epsilon N} \Vert h\Vert_{C^1(I)}
$$
for some constant $c=c(\Gamma) > 0.$ We rewrite
$$
\mathbb{E}_{\textbf{a}_2\in Z_{N_2}^{l,j}}\left(   \rho(\gamma_{\textbf{a}_1 \textbf{a}_2})^{-1} h(\gamma_{\textbf{a}_1 \textbf{a}_2} (x)) \right)
$$
as follows:
\begin{align*}
&\mathbb{E}_{\textbf{a}_2\in Z_{N_2}^{l,j}}\left[   \rho(\gamma_{\textbf{a}_1 \textbf{a}_2})^{-1}  \left( h(\gamma_{\textbf{a}_1 \textbf{a}_2} (x)) -  \mathbb{E}_{\textbf{b}_2\in Z_{N_2}^{l,j}}  h(\gamma_{\textbf{a}_1 \textbf{b}_2} (x))   \right)      \right]\\
&\qquad \qquad + \mathbb{E}_{\textbf{a}_2\in Z_{N_2}^{l,j}} \left(  \rho(\gamma_{\textbf{a}_1 \textbf{a}_2})^{-1} \mathbb{E}_{\textbf{b}_2\in Z_{N_2}^{l,j}}  h(\gamma_{\textbf{a}_1 \textbf{b}_2} (x)) \right)
\end{align*}
Applying the triangle inequality yields
\begin{align*}
&\left\Vert \mathbb{E}_{\textbf{a}_2\in Z_{N_2}^{l,j}}\left(   \rho(\gamma_{\textbf{a}_1 \textbf{a}_2})^{-1} h(\gamma_{\textbf{a}_1 \textbf{a}_2} (x)) \right) \right\Vert_V \leqslant  \\
&\quad \mathbb{E}_{\textbf{a}_2\in Z_{N_2}^{l,j}}   \left\Vert h(\gamma_{\textbf{a}_1 \textbf{a}_2} (x)) -  \mathbb{E}_{\textbf{b}_2\in Z_{N_2}^{l,j}}  h(\gamma_{\textbf{a}_1 \textbf{b}_2} (x))   \right\Vert_V  +    \left\Vert \mathbb{E}_{\textbf{a}_2\in Z_{N_2}^{l,j}} \left(  \rho(\gamma_{\textbf{a}_1 \textbf{a}_2})^{-1} \mathbb{E}_{\textbf{b}_2\in Z_{N_2}^{l,j}}  h(\gamma_{\textbf{a}_1 \textbf{b}_2} (x)) \right)\right\Vert_V.
\end{align*}
Thus we have
$$
\Sigma_l \leqslant \Sigma_l^{(1)} + \Sigma_l^{(2)},
$$
where
$$
\Sigma_l^{(1)} \defeq \mathbb{E}_{\textbf{a}_1\in W_{N_1}^l} \mathbb{E}_{\textbf{a}_2\in Z_{N_2}^{l,j}}   \left\Vert h(\gamma_{\textbf{a}_1 \textbf{a}_2} (x)) -  \mathbb{E}_{\textbf{b}_2\in Z_{N_2}^{l,j}}  h(\gamma_{\textbf{a}_1 \textbf{b}_2} (x))   \right\Vert_V
$$
and
$$
\Sigma_l^{(2)} \defeq \mathbb{E}_{\textbf{a}_1\in W_{N_1}^l} \left\Vert \mathbb{E}_{\textbf{a}_2\in Z_{N_2}^{l,j}} \left(  \rho(\gamma_{\textbf{a}_1 \textbf{a}_2})^{-1} \mathbb{E}_{\textbf{b}_2\in Z_{N_2}^{l,j}}  h(\gamma_{\textbf{a}_1 \textbf{b}_2} (x)) \right) \right\Vert_V.
$$
First we consider $\Sigma_l^{(1)}$. From the definition of $Z_{n}^{l,j}$ it follows that for all $\textbf{a}_2, \textbf{b}_2\in Z_{n_2}^{l,j}$, both $\gamma_{\textbf{a}_2}(x)$ and $\gamma_{\textbf{b}_2}(y)$ belong to the same interval $I_l$. For any pair of points $x_1,x_2$ in same interval, the fundamental theorem of calculus yields
$$
\Vert h(x_1)-h(x_2)\Vert_V = \left\Vert  \int_{x_1}^{x_2} h'(u)du \right\Vert_V \leqslant \vert x_1 - x_2\vert\cdot \Vert h'\Vert_\infty.
$$
Therefore, given $ \textbf{a}_1\in W_{n_1}^l$ and $\textbf{a}_2, \textbf{b}_2\in Z_{n_2}^{l,j}$, we obtain the estimate
\begin{align*}
\left\Vert h(\gamma_{\textbf{a}_1} (\gamma_{ \textbf{a}_2} (x))) -  h(\gamma_{\textbf{a}_1} (\gamma_{\textbf{b}_2} (x)))  \right\Vert_V 
&\leqslant \vert \gamma_{\textbf{a}_1} (\gamma_{ \textbf{a}_2} (x)) -  \gamma_{\textbf{a}_1} (\gamma_{\textbf{b}_2} (x)) \vert \Vert h'\Vert_\infty\\
&\leqslant \vert \gamma_{ \textbf{a}_2} (x) -  \gamma_{\textbf{b}_2} (x) \vert \sup_{z\in D_l} \vert \gamma_{\textbf{a}_1}'(z)\vert  \Vert h'\Vert_\infty\\
&\ll \Vert \gamma_{\textbf{a}_1}'\Vert_{\infty,I_l} \Vert h'\Vert_\infty.
\end{align*}
Thus,
\begin{align*}
\Sigma_l^{(1)} 
&\leqslant \mathbb{E}_{\textbf{a}_1\in W_{N_1}^l}  \mathbb{E}_{\textbf{a}_2\in Z_{N_2}^{l,j}} \mathbb{E}_{\textbf{b}_2\in Z_{N_2}^{l,j}} \left\Vert h(\gamma_{\textbf{a}_1 \textbf{a}_2} (x)) -  h(\gamma_{\textbf{a}_1 \textbf{b}_2} (x))  \right\Vert_V\\
&= \mathbb{E}_{\textbf{a}_1\in W_{N_1}^l}  \mathbb{E}_{\textbf{a}_2\in Z_{N_2}^{l,j}} \mathbb{E}_{\textbf{b}_2\in Z_{N_2}^{l,j}} \left\Vert h(\gamma_{\textbf{a}_1} (\gamma_{ \textbf{a}_2} (x))) -  h(\gamma_{\textbf{a}_1} (\gamma_{\textbf{b}_2} (x)))  \right\Vert_V\\
&\ll \mathbb{E}_{\textbf{a}_1\in W_{N_1}^l}  \mathbb{E}_{\textbf{a}_2\in Z_{N_2}^{l,j}} \mathbb{E}_{\textbf{b}_2\in Z_{N_2}^{l,j}} \Vert \gamma_{\textbf{a}_1}'\Vert_{\infty, I_l} \Vert h'\Vert_\infty\\
&= \mathbb{E}_{\textbf{a}_1\in W_{N_1}^l} \Vert \gamma_{\textbf{a}_1}'\Vert_{\infty,I_l} \Vert h'\Vert_\infty.
\end{align*}
Using the pressure estimate (Lemma \ref{pressure_lemma}) we obtain
\begin{equation}\label{S_1_final}
S_l^{(1)} \ll \frac{e^{N_1 P(1)}}{\vert W_{N_1}^l \vert} \Vert h'\Vert_\infty \ll e^{-c' N_1} \Vert h'\Vert_\infty,
\end{equation}
with $c' \defeq -P(1) +P(0)>0.$ 

Let us now concentrate on $\Sigma_l^{(2)}.$ Note that
\begin{align*}
\Sigma_l^{(2)} &= \mathbb{E}_{\textbf{a}_1\in W_{N_1}^l} \left\Vert \mathbb{E}_{\textbf{a}_2\in Z_{N_2}^{l,j}} \left( \rho(\gamma_{\textbf{a}_2})^{-1} \rho(\gamma_{\textbf{a}_1})^{-1} \mathbb{E}_{\textbf{b}_2\in Z_{N_2}^{l,j}}  h(\gamma_{\textbf{a}_1 \textbf{b}_2} (x))  \right) \right\Vert_V\\
&= \mathbb{E}_{\textbf{a}_1\in W_{N_1}^l} \left\Vert \mathbb{E}_{\textbf{a}_2\in Z_{N_2}^{l,j}} \left( \rho(\gamma_{\textbf{a}_2})^{-1} \right) \cdot \rho(\gamma_{\textbf{a}_1})^{-1} \cdot  \left( \mathbb{E}_{\textbf{b}_2\in Z_{N_2}^{l,j}}  h(\gamma_{\textbf{a}_1 \textbf{b}_2} (x))  \right) \right\Vert_V.
\end{align*}
The triangle inequality implies that
$$
\left\Vert \mathbb{E}_{\textbf{b}_2\in Z_{N_2}^{l,j}}  h(\gamma_{\textbf{a}_1 \textbf{b}_2} (x)) \right\Vert_V \leqslant \Vert h\Vert_{\infty}.
$$ 
Since $\rho$ is unitary we obtain
\begin{align*}
\Sigma_l^{(2)}
&\leqslant \mathbb{E}_{\textbf{a}_1\in W_{N_1}^l} \left\Vert \mathbb{E}_{\textbf{a}_2\in Z_{N_2}^{l,j}} \left( \rho(\gamma_{\textbf{a}_2})^{-1}\right)   \right\Vert      \left\Vert \mathbb{E}_{\textbf{b}_2\in Z_{N_2}^{l,j}}  h(\gamma_{\textbf{a}_1 \textbf{b}_2} (x)) \right\Vert_V\\
&\leqslant \mathbb{E}_{\textbf{a}_1\in W_{N_1}^l} \left\Vert \mathbb{E}_{\textbf{a}_2\in Z_{N_2}^{l,j}} \left( \rho(\gamma_{\textbf{a}_2})^{-1}\right)   \right\Vert      \Vert h\Vert_{\infty}\\
&= \mathbb{E}_{\textbf{a}_1\in W_{N_1}^l} \left\Vert \mathbb{E}_{\textbf{a}_2\in Z_{N_2}^{l,j}} \left( \rho(\gamma_{\textbf{a}_2})\right)   \right\Vert      \Vert h\Vert_{\infty}.
\end{align*}
In the last line we used that for any $\textbf{A}\in \mathrm{End}(W)$ the operator norm satisfies $\Vert \textbf{A}^\ast\Vert = \Vert \textbf{A}\Vert$, where $\textbf{A}^\ast$ denotes the adjoint of $\textbf{A}$, as well as
$$
\mathbb{E}_{\textbf{a}_2\in Z_{N_2}^{l,j}} \left( \rho(\gamma_{\textbf{a}_2})^{-1}\right) = \mathbb{E}_{\textbf{a}_2\in Z_{N_2}^{l,j}} \rho(\gamma_{\textbf{a}_2})^{\ast} = \left( \mathbb{E}_{\textbf{a}_2\in Z_{N_2}^{l,j}} \rho(\gamma_{\textbf{a}_2}) \right)^{\ast},
$$
which is once again a consequence of unitarity of $\rho.$ Hence, invoking the estimate \eqref{apply_cor}, we obtain
\begin{equation}\label{S_2_final}
\Sigma_l^{(2)} \ll  e^{- \frac{\epsilon}{6} N_2} \Vert h\Vert_{\infty}.
\end{equation}
Putting together \eqref{S_1_final} and \eqref{S_2_final} yields
\begin{equation}
\Sigma_l  \leqslant \Sigma_l^{(1)} + \Sigma_l^{(2)}  \ll e^{- \frac{\epsilon}{6} N_2} \Vert h\Vert_{\infty} + e^{-c'N_1}\Vert h'\Vert_\infty. 
\end{equation}
Finally, taking $N_1 = \lceil  \frac{\epsilon}{6c'+\epsilon} N \rceil $ and $N_2 = N-N_1, $ gives
\begin{equation}\label{finali}
\Sigma_l \ll e^{-  \frac{c'}{6c' + \epsilon} \epsilon N} \left( \Vert h\Vert_{\infty} + \Vert h'\Vert_\infty \right) = e^{- \frac{c'}{6c' + 1} \epsilon N} \Vert h\Vert_{C^1}
\end{equation}
for all $l\in [2m]$, completing the proof.
\end{proof}

\section{Proof of Theorem \ref{main_theorem}}\label{sec:part2mainthm}

We are now prepared to prove our main theorem. Fix a non-elementary Schottky group $\Gamma$ with Schottky generating set $\gamma_1 , \dots, \gamma_m$ as in §\ref{Schottky_groups}, fix a finite-dimensional, unitary representation $\rho\colon \Gamma\to \mathrm{U}(V),$ and put
$$
T(\rho) = \frac{1}{2m} \sum_{j=1}^{2m} \left(  \rho(\gamma_j) + \rho(\gamma_j)^{-1}\right).
$$
In view of the Venkov--Zograf induction formula \eqref{VZ_ind_formula} and Lemma \ref{lem:rel_trho_exp}, Theorem \ref{main_theorem} is a direct consequence of the following more general result:

\begin{theorem}\label{thm:second_part} Let notations and assumptions be as above. Assume that all the eigenvalues of $T(\rho)$ are inside the interval $[-1+\epsilon, 1-\epsilon]$ for some $\epsilon > 0$. Then, $Z_\Gamma(s,\rho)$ has no zeros in the half-plane $\mathrm{Re}(s)\geqslant \delta-\exp(-c \epsilon^{-1}),$ where the constant $c>0$ depends only on $\Gamma.$
\end{theorem}

\begin{proof}
Fix $\epsilon > 0$ and a finite-dimensional unitary representation $(\rho,V)$ of $\Gamma $ such that all the eigenvalues of $T(\rho)$ are inside $[-1+\epsilon, 1-\epsilon]$. Suppose that $Z_\Gamma(s,\rho)$ has a complex zero $s=\sigma + it$ satisfying $\sigma \geqslant \delta-\eta$. Our goal is to arrive at a contradiction with $\eta = \exp(-c \epsilon^{-1})$ for some positive constant $c = c(\Gamma) > 0$. Throughout the proof, all implied constants depend solely upon $\Gamma$.

Let $\eta_0 >0$ be the constant from Corollary \ref{cor:dist_zero_free}. We may assume that $\eta\in (0,\eta_0),$ in which case
\begin{equation}\label{tBoundApplUni}
\vert t\vert = O((\delta-\sigma)^{1/3}) = O(\eta^{1/3}).
\end{equation}
(The exponent $1/3$ may be replaced by any exponent strictly smaller than $1/2$, but $1/3$ is enough for our purposes.) 
 
Now we argue as in §\ref{sec:finishingDistZero}. If $Z_\Gamma(s,\rho) = 0$ then $\mathcal{L}_{s,\rho}$ has some non-trivial 1-eigenfunction, that is, there exists some $0\neq F\in H^2(D,V)$ satisfying $\mathcal{L}_{s,\rho}F = F$. Put $f = \varphi_\sigma^{-1} F$, let $x\in I = \bigsqcup_{j\in [2m]} I_j$ be a maximizer of the function $x\mapsto \Vert f(x)\Vert_V$, and let $j\in [2m]$ be the index for which $x\in I_j.$ By Proposition \ref{prop:strict_convexity}, for all $N\in \mathbb{N}$ and for all $\textbf{a}\in W_N^j$, the norm of
$$
f(x) - \gamma_{\textbf{a}}'(x)^{it} \rho(\gamma_\textbf{a})^{-1} f(\gamma_\textbf{a}(x))
$$
is bounded from above by
$$
O( e^{C_0 N}\sqrt{\eta} \Vert f\Vert_\infty),
$$
where the implied constant depends only on $\Gamma$. We simply write this as
\begin{equation}\label{applyinStrictConv}
f(x) = \gamma_{\textbf{a}}'(x)^{it} \rho(\gamma_\textbf{a})^{-1} f(\gamma_\textbf{a}(x)) + O( e^{C_0 N}\sqrt{\eta} \Vert f\Vert_\infty).
\end{equation}
The uniform hyperbolicity property in Proposition \ref{prop:DE} implies that for all $\textbf{a}\in W_N^j,$ 
$$
\vert \gamma_{\textbf{a}}'(x)^{it} - 1 \vert = \vert e^{it\log \gamma_{\textbf{a}}'(x)} - 1 \vert \ll \vert t\vert \vert \log \gamma_{\textbf{a}}'(x)\vert \ll N \vert t\vert.
$$
Combining this with \eqref{tBoundApplUni} it follows that
$$
\gamma_{\textbf{a}}'(x)^{it} = 1 + O(N \vert t\vert) = 1 + O\left( N \eta^{1/3} \right),
$$
which when inserted into \eqref{applyinStrictConv} gives
$$
f(x) = \rho(\gamma_\textbf{a})^{-1} f(\gamma_\textbf{a}(x)) + O\left(  \left( N \eta^{1/3} + e^{C_0 N} \sqrt{\eta} \right) \Vert f\Vert_\infty  \right).
$$
Now we average this over all words $\textbf{a}\in W_N^j$:
\begin{equation}\label{afterAverage}
f(x) = \mathbb{E}_{\textbf{a}\in W_N^j} \left( \rho(\gamma_\textbf{a})^{-1} f(\gamma_\textbf{a}(x)) \right) + O\left(  \left( N \eta^{1/3} + e^{C_0 N} \sqrt{\eta} \right) \Vert f\Vert_\infty  \right).
\end{equation}
Taking the norm of $V$ on both sides leads to
$$
\Vert f\Vert_\infty \leqslant \left\Vert \mathbb{E}_{\textbf{a}\in W_N^j} \left( \rho(\gamma_\textbf{a})^{-1} f(\gamma_\textbf{a}(x)) \right) \right\Vert_V + O\left(  \left( N \eta^{1/3} + e^{C_0 N} \sqrt{\eta} \right) \Vert f\Vert_\infty  \right).
$$
Invoking Proposition \ref{prop:exp_dec_average} gives
\begin{equation}\label{aboveBound}
\Vert f\Vert_\infty = O\left(  e^{-c \epsilon N} \Vert f\Vert_{C^1} + \left( N \eta^{1/3} + e^{C_0 N}\sqrt{\eta}  \right) \Vert f\Vert_\infty  \right).
\end{equation}
By Lemma \ref{estimate_der_eigenfct} we have
$$
\Vert f\Vert_{C^1} = O(\Vert f\Vert_\infty),
$$
which when inserted into \eqref{aboveBound} yields
$$
\Vert f\Vert_\infty  = O\left( \left( e^{-c \epsilon N} +  N \eta^{1/3} + e^{C_0 N}\sqrt{\eta}  \right) \Vert f\Vert_\infty  \right).
$$
Finally, as $f\neq 0$, we can divide both sides by $\Vert f\Vert_\infty \neq 0$ which gives
$$
1  \leqslant C \left(  e^{-c \epsilon N} +  N \eta^{1/3} + e^{C_0 N}\sqrt{\eta} \right)
$$
for some constant $C = C(\Gamma)>0$. To arrive at the desired contradiction we can now choose $N = \lceil c_0 \epsilon^{-1}\rceil$ and $\eta = \exp(-c_1 \epsilon^{-1})$ with some sufficiently large constants $c_0 > 0$ and $c_1 > 0$ depending only on $C$ and $C_0$. The proof of Theorem \ref{thm:second_part} is complete.
\end{proof}

\normalem
\bibliography{uniform_final} 
\bibliographystyle{amsplain}

\end{document}